\newif \ifSINUM \SINUMfalse

\ifSINUM

\documentclass[review]{siamart251216} 

\else

\documentclass[a4paper,11pt]{article}
\usepackage[english]{babel}
\usepackage[utf8]{inputenc}
\usepackage[T1]{fontenc}
\usepackage{amsthm}

\fi

\usepackage{amsmath,amsfonts,amssymb}
\usepackage{mathtools,mathbbol}
\DeclareSymbolFontAlphabet{\mathbb}{AMSb}
\DeclareSymbolFontAlphabet{\mathbbl}{bbold}
\DeclareMathAlphabet{\mathbbmsl}{U}{bbm}{m}{sl}
\usepackage{upgreek}
\usepackage{nicefrac}
\usepackage{cancel}
\usepackage{comment}
\usepackage{enumitem}
\usepackage{physics}
\usepackage{newtxtext,newtxmath}
\usepackage{graphicx}

\ifSINUM

\hypersetup{
  pdftitle = {Equivalence of mixed and nonconforming methods on general polytopal partitions. Part I: Multiscale and projection methods},
  pdfauthor = {S. Lemaire}
}

\else

\usepackage[dvipsnames, svgnames]{xcolor}
\usepackage[hmargin={28mm,28mm},vmargin={30mm,35mm}]{geometry}
\usepackage{authblk}
\usepackage{hyperref}
\hypersetup{
  pdftitle = {Equivalence of mixed and nonconforming methods on general polytopal partitions. Part I: Multiscale and projection methods},
  pdfauthor = {S. Lemaire},
  colorlinks = true,                        
  linkcolor = OliveGreen,                   
  citecolor = violet,                       
  filecolor = blue,                         
  urlcolor = magenta                        
}

\fi

\usepackage{pgfplots}
\pgfplotsset{every axis/.append style={
                     label style={font=\Large},
                     tick label style={font=\Large},
                     legend style={font=\Large}
                     }
}
\usetikzlibrary{matrix}
\usetikzlibrary{arrows.meta}

\ifSINUM

\headers{Equivalence of mixed and nonconforming methods}{S. Lemaire}
\title{Equivalence of mixed and nonconforming methods on general polytopal partitions\\Part I: Multiscale and projection methods\thanks{Submitted to the editors 02/16/2026.}}
\author{Simon Lemaire\thanks{Univ.~Lille, CNRS, Inria, UMR 8524--Laboratoire Paul Painlev\'e, F-59000 Lille, France (\email{simon.lemaire@inria.fr}).}}

\else

\newcommand{\email}[1]{\href{mailto:#1}{#1}}

\title{Equivalence of mixed and nonconforming methods on general polytopal partitions\\Part I: Multiscale and projection methods}
\author[1]{Simon Lemaire\footnote{\email{simon.lemaire@inria.fr}}} 
\affil[1]{Univ.~Lille, CNRS, Inria, UMR 8524--Laboratoire Paul Painlev\'e, F-59000 Lille, France}

\fi

\ifSINUM

\else

\newtheorem{lemma}{Lemma}
\newtheorem{proposition}{Proposition}

\fi

\newtheorem{remark}{Remark}
\newtheorem{example}{Example}
\newtheorem{problem}{Problem}
\newcommand{\defi}{\coloneqq}

\newcommand{\car}{{\rm card}}
\renewcommand{\dim}{{\rm dim}}
\newcommand{\spa}{{\rm span}}

\renewcommand{\vec}[1]{\boldsymbol{#1}}
\newcommand{\til}[1]{\widetilde{#1}}
\renewcommand{\bar}[1]{\overline{#1}}

\newcommand{\N}{\mathbb{N}}
\newcommand{\R}{\mathbb{R}}

\DeclareMathOperator{\vRot}{{\bf rot}}
\DeclareMathOperator{\sRot}{rot}

\DeclareMathOperator{\Grad}{\bf grad}
\DeclareMathOperator{\Curl}{\bf curl}
\DeclareMathOperator{\Div}{div}

\newcommand{\tb}{\mathbbl{b}}

\newcommand{\bmin}{\mathrm{b}_\flat}
\newcommand{\bmax}{\mathrm{b}_\sharp}

\newcommand{\LL}[1][\Omega]{L^2(#1)}
\newcommand{\Hone}[1][\Omega]{H^1(#1)}
\newcommand{\Hzone}[1][\Omega]{H^1_0(#1)}
\newcommand{\tilHzone}[1][k]{\til{H}^{1,#1}_0(\Omega)}

\newcommand{\Hphalf}[1][\pKK]{H^{\frac{1}{2}}(#1)}
\newcommand{\Hzphalf}[1][\pKK]{H^{\frac{1}{2}}_0(#1)}
\newcommand{\Hmhalf}[1][\pKK]{H^{-\frac{1}{2}}(#1)}

\newcommand{\vLL}[1][\Omega]{\vec{L}^2(#1)}
\newcommand{\vHone}[1][\Omega]{\vec{H}^1(#1)}

\newcommand{\Hcurlzw}[1][\Omega]{\vec{H}(\Curl^{\vec{0}}_{\tb^{-1}} ; #1)}

\newcommand{\Hdiv}[1][\Omega]{\vec{H}(\Div ; #1)}
\newcommand{\Hzdiv}[1][\Omega]{\vec{H}_0(\Div ; #1)}

\newcommand{\DD}{\mathcal{D}}
\newcommand{\KK}{\mathcal{K}}
\newcommand{\pKK}{\partial\KK}
\newcommand{\FF}{\mathcal{F}}
\newcommand{\FFi}{\FF^{\rm i}}
\newcommand{\FFb}{\FF^{\rm b}}

\newcommand{\pK}{\partial K}
\newcommand{\vN}{\vec{n}}

\newcommand{\Pol}{\mathcal{P}}

\newcommand{\vPol}{\vec{\Pol}}

\newcommand{\Gol}{\vec{\mathcal{G}}}
\newcommand{\kGol}{\Gol_{\rm c}}
\newcommand{\Col}{\vec{\mathcal{C}}}
\newcommand{\kCol}{\Col_{\rm c}}

\newcommand{\CR}{\til{\mathcal{CR}}}

\newcommand{\RT}{\vec{\mathcal{RT}}}
\newcommand{\BDM}{\vec{\mathcal{BDM}}}
\newcommand{\Ne}{\vec{\mathcal{N}}}

\newcommand{\projK}[1][\ell-1]{\pi^{#1}_K}
\newcommand{\projKK}[1][\ell-1]{\pi^{#1}_{\KK}}
\newcommand{\projF}[1][k-1]{\pi^{#1}_F}
\newcommand{\projFFK}[1][k-1]{\pi^{#1}_{\FF_K}}

\newcommand{\vprojK}[1][\ell-2]{\vec{\pi}^{#1}_K}
\newcommand{\vprojKK}[1][\ell-2]{\vec{\pi}^{#1}_{\KK}}
\newcommand{\vprojGK}[1][\ell-2]{\vec{\pi}^{#1}_{\Gol,K}}

\newcommand{\projFluxK}[2][\ell]{\vec{\pi}_{#2,K}^{#1}}
\newcommand{\projFluxKK}[2][\ell]{\vec{\pi}_{#2,\KK}^{#1}}

\newcommand{\wprojFluxKK}[2][\ell]{\vec{\pi}_{#2,\KK}^{#1,\tb^{-1}}}

\newcommand{\dPot}[1][\ell,k]{U^{#1}}
\newcommand{\dPotb}[1][\ell,k]{X^{#1}}
\newcommand{\gPot}[1][\ell,k]{\til{U}^{#1}_0(\Omega)}
\newcommand{\gPotb}[1][\ell,k]{\til{X}^{#1}_0(\Omega)}

\newcommand{\dHzphalf}[1][k]{\til{H}^{\frac{1}{2},#1}_0(\pKK)}
\newcommand{\dHmhalf}[1][k]{H^{-\frac{1}{2},#1}(\pKK)}

\newcommand{\dFlu}[1][\ell,k]{\vec{\Sigma}^{#1}}
\newcommand{\gFlu}[1][\ell,k]{\dFlu[#1](\Omega)}

\newcommand{\wprojFluK}[1][\ell]{\vec{\pi}_{\vec{\Sigma},K}^{#1,\tb^{-1}}}
\newcommand{\wprojFluKK}[1][\ell]{\vec{\pi}_{\vec{\Sigma},\KK}^{#1,\tb^{-1}}}

\begin{document}

\maketitle

\ifSINUM

\begin{abstract}
  We study equivalence, in the context of a variable diffusion problem, between (conforming) mixed methods and (primal) nonconforming methods defined on potentially general polytopal partitions. In this first paper of a series of two, we focus on multiscale and projection methods. For multiscale methods, we establish the first-level equivalence between four different (oversampling-free) approaches, thereby broadening the results of~\cite{CFELV:22}. For projection methods, in turn, we provide a simple criterion (to be checked in practice) for primal/mixed well-posedness and equivalence to hold true. In the process, we also shed a new light on some self-stabilized hybrid methods. Part II of this work will address (general) polytopal element methods.
\end{abstract}

\else

\begin{abstract}
  We study equivalence, in the context of a variable diffusion problem, between (conforming) mixed methods and (primal) nonconforming methods defined on potentially general polytopal partitions. In this first paper of a series of two, we focus on multiscale and projection methods. For multiscale methods, we establish the first-level equivalence between four different (oversampling-free) approaches, thereby broadening the results of~\cite{CFELV:22}. For projection methods, in turn, we provide a simple criterion (to be checked in practice) for primal/mixed well-posedness and equivalence to hold true. In the process, we also shed a new light on some self-stabilized hybrid methods. Part II of this work will address (general) polytopal element methods.
\end{abstract}

\fi

\section{Introduction}

Let $\Omega$ be a domain in $\R^d$, $d\in\{2,3\}$, i.e., a bounded and connected Lipschitz open set of $\R^d$. We consider the following variable diffusion problem.
\begin{problem}[Strong] \label{pb:strong}
  Find the flux $\vec{\sigma}:\Omega\to\R^d$ and the potential $u:\Omega\to\R$ s.t.
  \setcounter{equation}{-1}
  \begin{subequations} \label{eq:strong}    
    \begin{alignat}{2}
      \tb^{-1}\vec{\sigma}-\Grad u &= \vec{0} &\qquad&\text{{\rm in} $\Omega$}, \label{eq:strong.a}
      \\
      -\Div\vec{\sigma} &= f &\qquad&\text{{\rm in} $\Omega$}, \label{eq:strong.b}
      \\
      u  &=  0 &\qquad&\text{{\rm on} $\partial \Omega$}, \label{eq:strong.c}
    \end{alignat}
  \end{subequations}
  where $f:\Omega\to\R$ is a given load function and $\tb:\Omega\to\R^{d\times d}_{\rm sym}$ a mobility tensor field.
\end{problem}
\noindent
Throughout this work, we assume that $f\in\LL$\footnote{The case $f\in H^{-1}(\Omega)$ could be treated similarly, provided a decomposition $f=g+\Div\vec{G}$ with $g\in\LL$ and $\vec{G}\in\vLL$ is available. Setting $\hat{\vec{\sigma}}\defi\vec{\sigma}+\vec{G}$, the couple $(\hat{\vec{\sigma}},u)$ is then solution to Problem~\ref{pb:strong} with right-hand side $(\vec{0}\leftarrow\tb^{-1}\vec{G},f\leftarrow g,0)$.}, and that there are $\bmax\geq\bmin >0$ such that, for almost every $\vec{x}\in\Omega$, there holds $\bmin|\vec{\xi}|^2\leq\tb(\vec{x})\vec{\xi}{\cdot}\vec{\xi}\leq\bmax|\vec{\xi}|^2$ for all $\vec{\xi}\in\R^d$.
For $X\subset\bar{\Omega}$, we let $({\cdot},{\cdot})_X$ irrespectively denote the standard $\LL[X]$ or $\vLL[X]\defi\LL[X]^d$ inner product, and $\|{\cdot}\|_X\defi\sqrt{({\cdot},{\cdot})_X}$ the corresponding norm.
Classically, the primal weak form of Problem~\ref{pb:strong} seeks a potential $u\in\Hzone$ satisfying
\begin{problem}[Primal] \label{pb:c.primal}
  Find $u\in\Hzone$ s.t.
  \begin{equation} \label{eq:c.primal}
    (\tb\Grad u,\Grad v)_{\Omega}=(f,v)_{\Omega}\qquad\forall v\in\Hzone.
  \end{equation}
\end{problem}
\noindent
In turn, the mixed weak form of Problem~\ref{pb:strong} seeks a couple flux/potential $(\vec{\sigma},u)\in\Hdiv\times\LL$ satisfying
\begin{problem}[Mixed] \label{pb:c.mixed}
  Find $(\vec{\sigma},u)\in\Hdiv\times\LL$ s.t.
  \begin{subequations} \label{eq:c.mixed}
    \begin{alignat}{2}
      (\tb^{-1}\vec{\sigma},\vec{\tau})_{\Omega}+(u,\Div\vec{\tau})_{\Omega} &= 0 &\qquad&\forall\vec{\tau}\in\Hdiv, \label{eq:c.mixed.a}
      \\
      -(\Div\vec{\sigma},v)_{\Omega} &= (f,v)_{\Omega} &\qquad&\forall v\in\LL. \label{eq:c.mixed.b}
    \end{alignat}
  \end{subequations}
\end{problem}

Primal/mixed equivalence crucially hinges on the concept of hybridization.
At the discrete level, the first hybridization of a mixed (finite element) method appeared in 1965 in the work of Fraeijs de Veubeke~\cite{FrdeV:65}, in the linear elasticity context.
At that time, though, hybridization was merely used as an algebraic trick, often combined with static condensation, to efficiently solve the saddle-point systems arising from mixed formulations.
Only 20 years later was it realized by Arnold and Brezzi~\cite{ArBre:85} that hybridization could be more than an algebraic ruse, and provide an equivalent rewriting of mixed methods as (primal) nonconforming schemes, with built-in improved consistency of the potential reconstruction.
Not only the saddle-point system can be replaced by an equivalent (condensed) SPD system~\cite{Marin:85}, but the solution potential to this new system is both more accurate (asymptotically) and more conforming (in a weak sense) than the original (fully discontinuous) potential.
The seminal work of Arnold and Brezzi was later generalized by Chen and Arbogast~\cite{ChenZ:93,ArbCh:95,ChenZ:96}, who introduced the concept of projection finite element method, along with an abstract well-posedness theory.
Further clarifications about the algebraic structure of hybridized mixed methods were also brought in~\cite{CoGop:04}.
Whereas the hybridization of mixed methods has been extensively studied in the finite element literature, much less attention was dedicated to the hybridization of primal (nonconforming) methods.
Of course, from an algebraic point of view, such a rewriting is not particularly interesting; nevertheless, one may want/need, in practice, to reconstruct a conforming flux out of the solution to the primal problem.
Hybridizing the primal form of the problem gives rise to a flux-based formulation, in the sense that the globally coupled face unknowns of the resulting linear system are of flux type (cf.~\cite{RaTho:77a}).
However, starting from a given nonconforming finite element method (think, e.g., of the standard Crouzeix--Raviart method~\cite{CroRa:73}), it is usually not straightforward to recover an equivalent mixed method, because of the (probable) insufficient number of bubbles spanning the potential space. 

In this work, we aim to revisit hybridization techniques in the context of general polytopal grids. To this aim, we adopt a higher-level perspective based on virtual spaces. We present a finite-dimensional virtual construction for which primal/mixed equivalence is shown to hold true. Starting from the standard primal and mixed forms of the problem, which account for two distinct variational principles (based, respectively, on the potential and complementary energies), we hybridize them to obtain two additional formulations, the first one being flux-based (globally coupled face unknowns of flux type) and the second one being potential-based (globally coupled face unknowns of potential type). Primal/mixed equivalence is then established following the diagram depicted on Figure~\ref{fig:c.diag} (which is relative to the continuous case). Our virtual construction is valid for locally varying mobility tensors, thereby enabling us to build bridges between some (oversampling-free) multiscale approaches. In the primal and hybridized primal settings, we recover the equivalence result of~\cite{CFELV:22} between the first-level MsHHO and MHM methods. In the mixed and hybridized mixed settings, in comparison, our first-level equivalence result is novel (even the methods seem to be, to some extent). In the second part of this article, we assume that the mobility tensor is cell-wise constant (yet anisotropic), and we investigate practical situations in which our virtual construction is amenable to a practical implementation as a one-level method. It turns out that there is essentially one single situation in which the virtual construction is directly usable, which is the bubble-enriched nonconforming-$\Pol^1$/lowest-order Raviart--Thomas potential/flux pair (on simplices or hyperrectangles). In passing, we demonstrate that this equivalence is valid on arbitrary simplices (without any geometrical restrictions), which had seemingly never been reported in the literature (cf.~\cite[Sec.~4]{ChenZ:93} and~\cite[Sec.~10]{ArbCh:95}). Finally, we investigate situations in which the virtual construction cannot be computed. It leads us to rediscover projection finite element methods, extending the latter concept (under the shortened name of projection methods) to virtual potential spaces. This allows us, in turn, to reinterpret some self-stabilized hybrid high-order or weak Galerkin methods as projection schemes, thereby providing mixed finite element equivalents for them. We also derive, in the process, a simple criterion (to be checked in practice), valid on general polytopal meshes, to ensure primal/mixed (uniform) well-posedness and equivalence. This criterion is essentially a reformulation of the criterion of~\cite[Thm.~4]{ArbCh:95}.

The material of this article is organized as follows. In Section~\ref{se:cont}, we hybridize Problems~\ref{pb:c.primal} and~\ref{pb:c.mixed}, and we establish primal/mixed equivalence at the continuous level. In Section~\ref{se:fdc}, we introduce a finite-dimensional virtual counterpart of the construction of Section~\ref{se:cont}, and discuss its implications in terms of primal/mixed equivalence for first-level multiscale methods. In Section~\ref{se:proj}, we investigate the practical implementation of the latter construction, revisiting the concept of projection method and drawing links with some self-stabilized hybrid methods.

\section{An illuminating observation} \label{se:cont}

Let us assume that we are given a Lipschitz partition of the domain $\Omega$. More precisely, there exists a finite collection $\KK\defi\{K\}$ of (non-empty, disjoint) Lipschitz open subdomains $K\subset\Omega$ such that $\bigcup_{K\in\KK}\bar{K}=\bar{\Omega}$. In what follows, we let $\pKK\defi\bigcup_{K\in\KK}\pK$ denote the skeleton of the partition. For all $K\in\KK$, we also let $\vN_{\pK}:\pK\to\R^d$ denote the (a.e.~defined) unit normal vector field to $\pK$ pointing outward from $K$. We introduce on the partition $\KK$ the following broken functional spaces:
\begin{align*}
  \Hdiv[\KK]&\defi\left\{\vec{\tau}_{\KK}\in\vLL\mid\vec{\tau}_K\defi\vec{\tau}_{\KK\mid K}\in\Hdiv[K]\;\forall K\in\KK\right\},
  \\
  \Hone[\KK]&\defi\left\{v_{\KK}\in\LL\mid v_K\defi v_{\KK\mid K}\in\Hone[K]\;\forall K\in\KK\right\},
\end{align*}
as well as the broken operators $\Div_\KK:\Hdiv[\KK]\to\LL$ and $\Grad_\KK:\Hone[\KK]\to\vLL$ such that $(\Div_\KK\vec{\tau}_{\KK})_{\mid K}\defi\Div\vec{\tau}_K$ and $(\Grad_\KK v_{\KK})_{\mid K}\defi\Grad v_K$ for all $K\in\KK$.
We also introduce the following skeletal trace spaces:
\begin{align}
  \Hmhalf&\defi\left\{\tau_{\pKK}\defi(\tau_{\pK})_{K\in\KK}\in\prod_{K\in\KK}\Hmhalf[\pK]
  \left |\!
  \begin{array}{l}
    \exists\,\vec{\tau}\in\Hdiv\text{ s.t.}
    \\
    \tau_{\pK} = \vec{\tau}_{\mid\pK}{\cdot}\vN_{\pK} \;\forall K \in\KK
  \end{array}
  \right .
  \!\!\!\right\},\label{eq:c.Hmhalf}
  \\
  \Hzphalf&\defi\left\{v_{\pKK}\defi(v_{\pK})_{K\in\KK}\in\prod_{K\in\KK}\Hphalf[\pK]
  \left |\!
  \begin{array}{l}
    \exists\,v \in\Hzone\text{ s.t.}
    \\
    v_{\pK} = v_{\mid\pK} \;\forall K \in\KK
  \end{array}
  \right .
  \!\!\!\right\}.\label{eq:c.Hzphalf}
\end{align}
For all $\tau_{\pKK}\in\prod_{K\in\KK}\Hmhalf[\pK]$ and $v_{\pKK}\in\prod_{K\in\KK}\Hphalf[\pK]$, we define
\[\langle\tau_{\pKK},v_{\pKK}\rangle_{\pKK}\defi\sum_{K\in\KK}\langle\tau_{\pK},v_{\pK}\rangle_{\pK},\]
where $\langle\cdot,\cdot\rangle_{\pK}$ stands for the standard duality pairing between $\Hmhalf[\pK]$ and $\Hphalf[\pK]$.
Interestingly, for all $\tau_{\pKK}\in\Hmhalf$ and $v_{\pKK}\in\Hzphalf$, there holds
\begin{equation} \label{eq:c.fund}
  \langle\tau_{\pKK},v_{\pKK}\rangle_{\pKK}=\sum_{K\in\KK}\big((\Div\vec{\tau},v)_K+(\vec{\tau},\Grad v)_K\big)=(\Div\vec{\tau},v)_{\Omega}+(\vec{\tau},\Grad v)_{\Omega}=0.
\end{equation}
In what follows, for $\vec{\tau}_\KK\in\Hdiv[\KK]$, we will denote by $\vec{\tau}_{\KK\mid\pKK}{\cdot}\vN_{\pKK}\in\prod_{K\in\KK}\Hmhalf[\pK]$ the collection $(\vec{\tau}_{K\mid\pK}{\cdot}\vN_{\pK})_{K\in\KK}$.
Similarly, for $v_\KK\in\Hone[\KK]$, $v_{\KK\mid\pKK}\in\prod_{K\in\KK}\Hphalf[\pK]$ will henceforth denote the collection $(v_{K\mid\pK})_{K\in\KK}$.
The space characterizations below hold true:
\begin{align}
  \Hdiv&=\big\{\vec{\tau}_\KK\in\Hdiv[\KK]\mid\langle\vec{\tau}_{\KK\mid\pKK}{\cdot}\vN_{\pKK}, v_{\pKK}\rangle_{\pKK}=0\;\;\forall v_{\pKK}\in\Hzphalf\big\},\label{eq:c.Hdiv}\\
  \Hzone&=\big\{v_\KK\in\Hone[\KK]\mid\langle\tau_{\pKK},v_{\KK\mid\pKK}\rangle_{\pKK}=0\;\;\forall\tau_{\pKK}\in\Hmhalf\big\}.\label{eq:c.Hzone}
\end{align}

With the partition $\KK$ at hand, it is possible to introduce two additional weak forms for Problem~\ref{pb:strong}. These two novel variational formulations will be referred to in the sequel as {\em hybridized}. They have the particularity to lie in-between primal and mixed, in the sense that they share defining features of both types.
The hybridized primal weak form of Problem~\ref{pb:strong}, obtained from Problem~\ref{pb:c.primal} by relaxing along $\pKK$ the $\Hzone$-conformity constraint on $u$, seeks a couple potential/normal flux trace $(u_\KK,\sigma_{\pKK})\in\Hone[\KK]\times\Hmhalf$ satisfying
\begin{problem}[Hybridized primal] \label{pb:c.hprimal}
  Find $(u_\KK,\sigma_{\pKK})\in\Hone[\KK]\times\Hmhalf$ s.t.
  \begin{subequations} \label{eq:c.hprimal}
    \begin{alignat}{2}
      (\tb\Grad_\KK u_\KK,\Grad_\KK v_\KK)_{\Omega}-\langle\sigma_{\pKK},v_{\KK\mid\pKK}\rangle_{\pKK} &= (f,v_\KK)_{\Omega} &\qquad&\forall v_\KK\in\Hone[\KK], \label{eq:c.hprimal.a}
      \\
      \langle\tau_{\pKK},u_{\KK\mid\pKK}\rangle_{\pKK} &= 0 &\qquad&\forall\tau_{\pKK}\in\Hmhalf. \label{eq:c.hprimal.b}
    \end{alignat}
  \end{subequations}
\end{problem}
\noindent
In turn, the hybridized mixed weak form of Problem~\ref{pb:strong}, obtained from Problem~\ref{pb:c.mixed} by relaxing along $\pKK$ the $\Hdiv$-conformity constraint on $\vec{\sigma}$, seeks a triple flux/potential/potential trace $(\vec{\sigma}_\KK,u_\KK,u_{\pKK})\in\Hdiv[\KK]\times\LL\times\Hzphalf$ satisfying
\begin{problem}[Hybridized mixed] \label{pb:c.hmixed}
  Find $(\vec{\sigma}_\KK,u_\KK,u_{\pKK})\in\Hdiv[\KK]\times\LL\times\Hzphalf$ s.t.
  \begin{subequations} \label{eq:c.hmixed}
    \begin{alignat}{2}
      (\tb^{-1}\vec{\sigma}_\KK,\vec{\tau}_\KK)_{\Omega}+(u_\KK,\Div_\KK\vec{\tau}_\KK)_{\Omega}-\langle\vec{\tau}_{\KK\mid\pKK}{\cdot}\vN_{\pKK}, u_{\pKK}\rangle_{\pKK} &= 0 &\quad&\forall\vec{\tau}_\KK\in\Hdiv[\KK], \label{eq:c.hmixed.a}
      \\
      -(\Div_\KK\vec{\sigma}_\KK,v_\KK)_{\Omega} &= (f,v_\KK)_{\Omega} &\quad&\forall v_\KK\in\LL, \label{eq:c.hmixed.b}
      \\
      \langle\vec{\sigma}_{\KK\mid\pKK}{\cdot}\vN_{\pKK}, v_{\pKK}\rangle_{\pKK} &= 0&\quad&\forall v_{\pKK}\in\Hzphalf. \label{eq:c.hmixed.c}
    \end{alignat}
  \end{subequations}
\end{problem}
\noindent
Remark that the hybridized primal Problem~\ref{pb:c.hprimal} (resp.~the hybridized mixed Problem~\ref{pb:c.hmixed}) accounts for a variational principle based on the potential energy (resp.~the complementary energy), which is a defining (energetic) feature of primal (resp.~mixed) formulations, but at the same time its globally coupled face unknowns are of flux type (resp.~of potential type), which is a defining (algebraic) feature of mixed methods (resp.~primal methods).
It is an easy matter to realize that the two primal formulations (Problems~\ref{pb:c.primal} and~\ref{pb:c.hprimal}) are equivalent, which is also true of the two mixed formulations (Problems~\ref{pb:c.mixed} and~\ref{pb:c.hmixed}). 
\begin{proposition}[Equivalence of primal forms] \label{pr:c.eqprimal}
  Problems~\ref{pb:c.primal} and~\ref{pb:c.hprimal} are equivalent:
  \begin{itemize}
    \item[$\Rightarrow$] if $u\in\Hzone$ solves Problem~\ref{pb:c.primal}, then there exists $\sigma_{\pKK}\in\Hmhalf$ such that the couple $(u_\KK\defi u,\sigma_{\pKK})\in\Hone[\KK]\times\Hmhalf$ solves Problem~\ref{pb:c.hprimal};
    \item[$\Leftarrow$] reciprocally, if the couple $(u_\KK,\sigma_{\pKK})\in\Hone[\KK]\times\Hmhalf$ solves Problem~\ref{pb:c.hprimal}, then there holds $u_\KK\in\Hzone$, and $u\defi u_\KK\in\Hzone$ solves Problem~\ref{pb:c.primal}.
  \end{itemize}
\end{proposition}
\begin{proof}
  We prove each implication separately.
  \begin{itemize}
    \item[$\Rightarrow$] Since, by~\eqref{eq:c.Hzone}, $\Hzone=\big\{v_\KK\in\Hone[\KK]\mid\langle\tau_{\pKK},v_{\KK\mid\pKK}\rangle_{\pKK}=0\;\forall\tau_{\pKK}\in\Hmhalf\big\}$, the result follows from the theory of Lagrange multipliers.
    \item[$\Leftarrow$] From~\eqref{eq:c.hprimal.b} and~\eqref{eq:c.Hzone}, we first infer that $u_\KK\in\Hzone$. We conclude testing~\eqref{eq:c.hprimal.a} with $v_\KK\in\Hzone\subset\Hone[\KK]$, and invoking again~\eqref{eq:c.Hzone} to cancel out the term $\langle\sigma_{\pKK},v_{\KK\mid\pKK}\rangle_{\pKK}$. 
  \end{itemize}
\end{proof}
\begin{proposition}[Equivalence of mixed forms] \label{pr:c.eqmixed}
  Problems~\ref{pb:c.mixed} and~\ref{pb:c.hmixed} are equivalent:
  \begin{itemize}
    \item[$\Rightarrow$] if the couple $(\vec{\sigma},u)\in\Hdiv\times\LL$ solves Problem~\ref{pb:c.mixed}, then there is $u_{\pKK}\in\Hzphalf$ such that the triple $(\vec{\sigma}_\KK\defi\vec{\sigma},u_\KK\defi u,u_{\pKK})\in\Hdiv[\KK]\times\LL\times\Hzphalf$ solves Problem~\ref{pb:c.hmixed};
    \item[$\Leftarrow$] reciprocally, if the triple $(\vec{\sigma}_\KK,u_\KK,u_{\pKK})\in\Hdiv[\KK]\times\LL\times\Hzphalf$ solves Problem~\ref{pb:c.hmixed}, then there holds $\vec{\sigma}_\KK\in\Hdiv$, and the couple $(\vec{\sigma}\defi\vec{\sigma}_\KK,u\defi u_\KK)\in\Hdiv\times\LL$ solves Problem~\ref{pb:c.mixed}.
  \end{itemize}
\end{proposition}
\begin{proof}
  We prove each implication separately.
  \begin{itemize}
    \item[$\Rightarrow$] Since, by~\eqref{eq:c.Hdiv}, $\Hdiv=\big\{\vec{\tau}_\KK\in\Hdiv[\KK]\mid\langle\vec{\tau}_{\KK\mid\pKK}{\cdot}\vN_{\pKK}, v_{\pKK}\rangle_{\pKK}=0\;\forall v_{\pKK}\in\Hzphalf\big\}$, the result follows from the theory of Lagrange multipliers.
    \item[$\Leftarrow$] From~\eqref{eq:c.hmixed.c} and~\eqref{eq:c.Hdiv}, we first infer that $\vec{\sigma}_\KK\in\Hdiv$. We conclude testing~\eqref{eq:c.hmixed.a} with $\vec{\tau}_\KK\in\Hdiv\subset\Hdiv[\KK]$, and invoking again~\eqref{eq:c.Hdiv} to cancel out the term $\langle\vec{\tau}_{\KK\mid\pKK}{\cdot}\vN_{\pKK}, u_{\pKK}\rangle_{\pKK}$. 
  \end{itemize}
\end{proof}
We now aim at drawing a path (i) from primal to mixed formulations on the one side, and (ii) from mixed to primal formulations on the other side. To perform so, we prove below that Problem~\ref{pb:c.hprimal} (hybridized primal) implies Problem~\ref{pb:c.mixed} (mixed), and that Problem~\ref{pb:c.hmixed} (hybridized mixed) implies Problem~\ref{pb:c.primal} (primal).
\begin{lemma}[From primal to mixed] \label{le:c.ptom}
  Assume the couple $(u_\KK,\sigma_{\pKK})\in\Hone[\KK]\times\Hmhalf$ solves Problem~\ref{pb:c.hprimal}. Then, there holds $\tb\Grad_{\KK}u_\KK\in\Hdiv$, and the couple $(\vec{\sigma}\defi\tb\Grad_{\KK}u_\KK,u\defi u_\KK)\in\Hdiv\times\LL$ solves Problem~\ref{pb:c.mixed}.
\end{lemma}
\begin{proof}
  The result falls in two steps.\\
  (a) First, in each subdomain $K\in\KK$, we test~\eqref{eq:c.hprimal.a} with $v_K\in\Hone[K]$. For all $v_K\in\Hzone[K]$, there holds $(\tb\Grad u_K,\Grad v_K)_K=(f,v_K)_K$, thereby $\vec{\sigma}_K\defi\tb\Grad u_K\in\Hdiv[K]$. We can thus integrate by parts the term $(\vec{\sigma}_K,\Grad v_K)_K$, yielding
  \[-(\Div\vec{\sigma}_K,v_K)_K+\langle\vec{\sigma}_{K\mid\pK}{\cdot}\vN_{\pK}-\sigma_{\pK},v_{K\mid\pK}\rangle_{\pK}=(f,v_K)_K.\]
  Since the above identity is valid for any $v_K\in\Hone[K]$, we both infer that $-\Div\vec{\sigma}_K=f$ a.e.~in $K$, and that $\vec{\sigma}_{K\mid\pK}{\cdot}\vN_{\pK}=\sigma_{\pK}$ in $\Hmhalf[\pK]$. Then, because $\sigma_{\pKK}\in\Hmhalf$, it follows that $\vec{\sigma}_\KK=\tb\Grad_\KK u_\KK\in\Hdiv$, and that $-(\Div\vec{\sigma}_\KK,v)_{\Omega}=(f,v)_{\Omega}$ for all $v\in\LL$.\\
  (b) Second, we prove that $(\tb^{-1}\vec{\sigma}_\KK,\vec{\tau})_{\Omega}+(u_\KK,\Div\vec{\tau})_{\Omega}=0$ for all $\vec{\tau}\in\Hdiv$.
  The conclusion follows setting $(\vec{\sigma}\defi\vec{\sigma}_\KK,u\defi u_\KK)\in\Hdiv\times\LL$.
  For any $\vec{\tau}\in\Hdiv$, letting $\tau_{\pKK}\in\prod_{K\in\KK}\Hmhalf[\pK]$ be such that $\tau_{\pKK}\defi\vec{\tau}_{\mid\pKK}{\cdot}\vN_{\pKK}$, there holds $\tau_{\pKK}\in\Hmhalf$. Then, recalling that $\vec{\sigma}_K=\tb\Grad u_K$ for all $K\in\KK$, and invoking~\eqref{eq:c.hprimal.b}, we infer
  \begin{align*}
    (\tb^{-1}\vec{\sigma}_\KK,\vec{\tau})_{\Omega}+(u_\KK,\Div\vec{\tau})_{\Omega}&=\sum_{K\in\KK}\big((\Grad u_K,\vec{\tau})_K+(u_K,\Div\vec{\tau})_K\big)\\
    &=\sum_{K\in\KK}\langle\tau_{\pK},u_{K\mid\pK}\rangle_{\pK}=\langle\tau_{\pKK},u_{\KK\mid\pKK}\rangle_{\pKK}=0.
  \end{align*}
\end{proof}
\begin{lemma}[From mixed to primal] \label{le:c.mtop}
  Assume the triple $(\vec{\sigma}_\KK,u_\KK,u_{\pKK})\in\Hdiv[\KK]\times\LL\times\Hzphalf$ solves Problem~\ref{pb:c.hmixed}. Then, there holds $u_\KK\in\Hzone$, and $u\defi u_\KK\in\Hzone$ solves Problem~\ref{pb:c.primal}.
\end{lemma}
\begin{proof}
  The result also falls in two steps.\\
  (a) First, in each subdomain $K\in\KK$, we test~\eqref{eq:c.hmixed.a} with $\vec{\tau}_K\in\Hdiv[K]$. For all $\vec{\tau}_K\in\Hzdiv[K]$, there holds $(\tb^{-1}\vec{\sigma}_K,\vec{\tau}_K)_K+(u_K,\Div\vec{\tau}_K)_K=0$, thereby $u_K\in\Hone[K]$. We can thus integrate by parts the term $(u_K,\Div\vec{\tau}_K)_K$, yielding
  \[(\tb^{-1}\vec{\sigma}_K-\Grad u_K,\vec{\tau}_K)_K+\langle\vec{\tau}_{K\mid\pK}{\cdot}\vN_{\pK}, u_{K\mid\pK}-u_{\pK}\rangle_{\pK}=0.\]
  Since the above identity is valid for any $\vec{\tau}_K\in\Hdiv[K]$, we both infer that $\tb^{-1}\vec{\sigma}_K=\Grad u_K$ a.e.~in $K$, and $u_{K\mid\pK}=u_{\pK}$ a.e.~on $\pK$. Given $u_{\pKK}\in\Hzphalf$, it follows that $u_\KK\in\Hzone$.\\
  (b) From~\eqref{eq:c.hmixed.c} and~\eqref{eq:c.Hdiv}, we directly infer that $\vec{\sigma}_\KK\in\Hdiv$, hence $\Div_\KK\vec{\sigma}_\KK=\Div(\tb\Grad u_\KK)$ (recall that $u_\KK\in\Hzone$). Now, testing~\eqref{eq:c.hmixed.b} with $v_\KK\defi v\in\Hzone\subset\LL$, and integrating its left-hand side by parts, we get that $(\tb\Grad u_\KK,\Grad v)_{\Omega}=(f,v)_{\Omega}$ for all $v\in\Hzone$.
  The conclusion follows setting $u\defi u_\KK\in\Hzone$.
\end{proof}

Let us summarize our developments so far. What we have just established above, following the diagram depicted on Figure~\ref{fig:c.diag}, is the equivalence between primal and mixed (weak) formulations of the problem.
Of particular importance are the two spaces $\Hzone$ and $\Hdiv$ in which lie, respectively, the solution potential $u$ and the solution flux $\vec{\sigma}=\tb\Grad u$. As a matter of fact, these two spaces encode the inter-subdomain conformity requirements that both the potential and the flux shall satisfy over the partition so as to be admissible.
If this primal/mixed equivalence result is natural (and expected) at the infinite-dimensional level, mimicking it at the finite-dimensional level is less obvious.
Several questions then arise:
\begin{itemize}
  \item[$\bullet$] is it possible to find 
    finite-dimensional counterparts of the spaces $\Hzone$ and $\Hdiv$ for which primal/mixed equivalence holds true at the discrete level?
  \item[$\bullet$] which degree of inter-subdomain conformity is it possible to preserve at the finite-dimensional (discrete) level for the potential and flux spaces?
  \item[$\bullet$] to what extent can the discrete counterparts of the solution potential and flux preserve the structure of the constitutive relation $\vec{\sigma}=\tb\Grad u$ from the model at hand?
\end{itemize}
In the next section (that is, Section~\ref{se:fdc}), we provide a partial answer to these questions. We indeed exhibit a finite-dimensional (virtual) construction for which primal/mixed equivalence holds true, and which exactly preserves the structure of the constitutive relation $\vec{\sigma}=\tb\Grad u$. This construction, which happens to be the cornerstone of a number of existing multiscale (a.k.a.~two-level) methods, also preserves full $\Hdiv$-conformity for the flux, as well as a weak $\Hzone$-conformity property for the potential. Unfortunately, and this topic is explored in Section~\ref{se:proj}, there are actually very few situations (in terms of domain partitioning and dimensionality) in which the latter construction is amenable to a practical implementation as a one-level method. We thus explore, in the subsequent developments, alternative (existing) one-level constructions, for which we establish primal/mixed equivalence, but for which we will see that they violate the structure of the constitutive law.
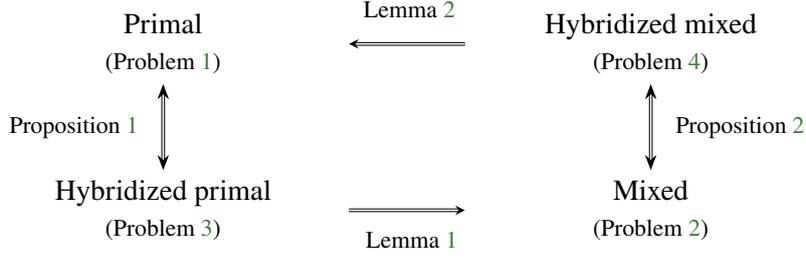
\begin{figure}[h]
  \centering
  \begin{tikzpicture}
    \matrix (m) [matrix of math nodes,row sep=3em,column sep=4em,minimum width=2em]
            {
            \begin{minipage}{0.3\textwidth}\centering\text{Primal}\\\text{{\footnotesize (Problem~\ref{pb:c.primal})}}\end{minipage} & \begin{minipage}{0.3\textwidth}\centering\text{Hybridized mixed}\\\text{{\footnotesize (Problem~\ref{pb:c.hmixed})}}\end{minipage} \\
            \begin{minipage}{0.3\textwidth}\centering\text{Hybridized primal}\\\text{{\footnotesize (Problem~\ref{pb:c.hprimal})}}\end{minipage} & \begin{minipage}{0.3\textwidth}\centering\text{Mixed}\\\text{{\footnotesize (Problem~\ref{pb:c.mixed})}}\end{minipage}\\};
            \path[-stealth]
            (m-1-1) edge [double,Stealth-Stealth] node [left,xshift=-0.5em] {{\footnotesize \text{Proposition~\ref{pr:c.eqprimal}}}} (m-2-1)
            (m-2-1) edge [double] node [right,xshift=-1.75em,yshift=-1em] {{\footnotesize \text{Lemma~\ref{le:c.ptom}}}} (m-2-2)
            (m-1-2) edge [double] node [left,xshift=2em,yshift=1.2em] {{\footnotesize \text{Lemma~\ref{le:c.mtop}}}} (m-1-1)       
            (m-2-2) edge [double,Stealth-Stealth] node [right,xshift=0.5em] {\footnotesize {\text{Proposition~\ref{pr:c.eqmixed}}}} (m-1-2);
  \end{tikzpicture}
  \caption{Equivalence diagram (continuous case): formulations on the first row are potential-based, whereas those on the second row are flux-based.\label{fig:c.diag}}
\end{figure}
  
\section{A finite-dimensional counterpart} \label{se:fdc}

From now on, we assume that $\Omega$ is a (Lipschitz) $d$-polytope, and that the partition $\KK$ is itself (Lipschitz) $d$-polytopal.
From a more formal perspective, we shall consider discretizations $\DD\defi(\KK,\FF)$ of $\Omega$ in the sense of~\cite[Def.~1.4]{DPDro:20}.

\subsection{Polytopal discretizations}

The set $\KK\defi\{K\}$ is a finite collection of (disjoint) Lipschitz open $d$-polytopes $K$ (the mesh cells) such that $\bigcup_{K\in\KK}\bar{K}=\bar{\Omega}$.
The barycenter of any mesh cell $K\in\KK$ is denoted by $\bar{\vec{x}}_K$, its diameter by $h_K$.
In turn, the set $\FF\defi\{F\}$ is a finite collection of (disjoint) connected subsets $F$ of $\pKK\defi\bigcup_{K\in\KK}\pK$ (the mesh faces) such that $\bigcup_{F\in\FF}\bar{F}=\pKK$. Besides, for all $F\in\FF$,
\begin{enumerate}
  \item[(i)] $F$ is a Lipschitz, relatively open $(d{-}1)$-polytopal subset of an affine hyperplane;
  \item[(ii)] either there are two distinct mesh cells $K^+,K^-\in\KK$ s.t.~$\bar{F}\subseteq\pK^+\cap\pK^-$ ($F$ is an interface), or there is one mesh cell $K\in\KK$ s.t.~$\bar{F}\subseteq\pK\cap\partial\Omega$ ($F$ is a boundary face).
\end{enumerate}
Interfaces are collected in the set $\FFi$, whereas boundary faces are collected in the set $\FFb$.
Besides, for all $K\in\KK$, we denote by $\FF_K$ the subset of $\FF$ such that $\bigcup_{F\in\FF_K}\bar{F}=\pK$.
For all $K\in\KK$, we recall that $\vN_{\pK}:\pK\to\R^d$ denotes the unit normal vector field to $\pK$ pointing outward from $K$. For all $F\in\FF_K$, we further let $\vN_{K,F}\defi\vN_{\pK\mid F}$ be the (constant) unit normal vector to the hyperplane containing $F$ pointing outward from $K$.
For all $F\in\FF$, in turn, we define $\vN_F$ as the (constant) unit normal vector to $F$ such that either $\vN_F\defi\vN_{K^+,F}$ if $F\subset\pK^+\cap\pK^-\in\FFi$, or $\vN_F\defi\vN_{K,F}$ if $F\subset\pK\cap\partial\Omega\in\FFb$.
For further use, we finally let, for all $K\in\KK$ and $F\in\FF_K$, $\varepsilon_{K,F}\in\{-1,1\}$ be such that $\varepsilon_{K,F}\defi\vN_{K,F}{\cdot}\vN_F$.

In what follows, locally to any $K\in\KK$, we write $a\lesssim b$ (resp.~$a\gtrsim b$) in place of $a\leq Cb$ (resp.~$a\geq Cb$) if $C>0$ does not depend on $h_K$. Note that the constant $C$ may however depend on the chunkiness parameter (as per~\cite[Def.~1.9]{DPDro:20}) of the mesh cell $K$. When $b\lesssim a\lesssim  b$, we simply write $a\eqsim b$.

\subsection{Polynomial spaces}

For $n\in\N$ and $m\in\{d{-}1,d\}$, $\Pol^n_m$ denotes the linear space of $m$-variate polynomials of total degree at most $n$, with the convention that $\Pol^0_m$ is identified to $\R$, and that $\Pol^{-1}_m\defi\{0\}$.
For any $X\in\FF\cup\KK$, we define $\Pol^n(X)$ as the linear space spanned by the restrictions to $X$ of the polynomials in $\Pol^n_d$. For $X$ of Hausdorff dimension $m\in\{d{-}1,d\}$, $\Pol^n(X)$ is isomorphic to $\Pol^n_m$ (cf.~\cite[Prop.~1.23]{DPDro:20}). We denote by $\pi^n_X$ the $\LL[X]$-orthogonal projector onto $\Pol^n(X)$.
Let now $(\mathcal{U},\mathcal{X})\in\big\{(\pK,\FF_K),(\pKK,\FF),(\Omega,\KK)\big\}$ (for any $K\in\KK$), in such a way that $\bigcup_{X\in\mathcal{X}}\bar{X}=\bar{\mathcal{U}}$. We define the broken polynomial spaces
\[\Pol^n(\mathcal{X})\defi\left\{v_{\mathcal{X}}\in\LL[\mathcal{U}]\mid v_X\defi v_{\mathcal{X}\mid X}\in\Pol^n(X)\;\forall X\in\mathcal{X}\right\},\]
whose respective $\LL[\mathcal{U}]$-orthogonal projectors are denoted by $\pi^n_{\mathcal{X}}$.
Let us turn to vector-valued polynomial spaces. For any $K\in\KK$, we set $\vPol^n(K)\defi\Pol^n(K)^d$, and we let $\vprojK[n]$ denote the corresponding $\vLL[K]$-orthogonal projector. At the global level, we let $\vPol^n(\KK)\defi\Pol^n(\KK)^d$, with projector $\vprojKK[n]$. We further introduce, for any $K\in\KK$,
\begin{equation*}
  \begin{aligned}
    \Gol^n(K)&\defi\Grad\big(\Pol^{n+1}(K)\big),\qquad&\kGol^n(K)&\defi\vPol^{n-1}(K){\times}(\vec{x}-\bar{\vec{x}}_K),\\
    \Col^n(K)&\defi\Curl\big(\vPol^{n+1}(K)\big),\qquad&\kCol^n(K)&\defi\Pol^{n-1}(K)(\vec{x}-\bar{\vec{x}}_K).
  \end{aligned}
\end{equation*}
When $d=2$, (i) the space $\kGol^n(K)$ is to be understood as $\kGol^n(K)=\Pol^{n-1}(K)\,(\vec{x}-\bar{\vec{x}}_K)^\perp$, where $\vec{z}^\perp$ denotes the rotation of angle $-\pi/2$ of $\vec{z}$, and (ii) the space $\Col^n(K)$ is to be understood as $\Col^n(K)=\vRot\big(\Pol^{n+1}(K)\big)$, where $\vRot z\defi(\Grad z)^\perp$.
As a consequence of the homotopy formula~\cite[Thm.~7.1]{Arnol:18}, the (non $\vLL[K]$-orthogonal) direct sum decompositions below hold:
\begin{equation*} 
  \vPol^n(K)=\Gol^n(K)\oplus\kGol^n(K)=\Col^n(K)\oplus\kCol^n(K).
\end{equation*}
Remark that $\kGol^n(K)=\Col^{n-1}(K){\times}(\vec{x}-\bar{\vec{x}}_K)$ when $d=3$.
We will also need, in what follows, the local Raviart--Thomas and (first-kind) N\'ed\'elec spaces~\cite{RaTho:77b,Nedel:80}: for any $K\in\KK$, $n\in\N^\star$,
\begin{equation} \label{eq:RTN}
  \RT^n(K)\defi\Col^{n-1}(K)\oplus\kCol^n(K),\qquad\Ne^n(K)\defi\Gol^{n-1}(K)\oplus\kGol^n(K).
\end{equation}
Both spaces are strictly sandwiched between $\vPol^{n-1}(K)$ and $\vPol^n(K)$.
For $\vec{\mathcal{Q}}\in\{\Gol,\kGol,\Col,\kCol\}$ or $\vec{\mathcal{Q}}\in\{\RT,\Ne\}$, we let $\projFluxK[n]{\vec{\mathcal{Q}}}$ denote the $\vLL[K]$-orthogonal projector onto $\vec{\mathcal{Q}}^n(K)$.
At the global level, the broken versions of the above polynomial spaces are defined by
\[\vec{\mathcal{Q}}^n(\KK)\defi\left\{\vec{v}_\KK\in\vLL\mid\vec{v}_K\defi\vec{v}_{\KK\mid K}\in\vec{\mathcal{Q}}^n(K)\;\forall K\in\KK\right\},\]
with corresponding $\vLL$-orthogonal projectors denoted by $\projFluxKK[n]{\vec{\mathcal{Q}}}$.

\subsection{Finite-dimensional functional spaces}

Let $\ell,k\in\N^\star$ be given. We first define, for any $v_\KK\in\Hone[\KK]$, the jump $\llbracket v_\KK\rrbracket_F$ of $v_\KK$ across $F\in\FF$. If $F\subset\pK^+\cap\pK^-\in\FFi$, $\llbracket v_\KK\rrbracket_F\defi v_{K^+\mid F}-v_{K^-\mid F}=\sum_{K\in\{K^+,K^-\}}\varepsilon_{K,F}v_{K\mid F}$, whereas if $F\subset\pK\cap\partial\Omega\in\FFb$, $\llbracket v_\KK\rrbracket_F\defi v_{K\mid F}$. We then let
\begin{equation*}
  \tilHzone\defi\left\{v_{\KK}\in\Hone[\KK]\mid\projF(\llbracket v_\KK\rrbracket_F)\equiv 0\;\forall F\in\FF\right\}.
\end{equation*}
Remark that $\tilHzone\not\subset\Hzone$, which is emphasized in the notation by the use of a tilde. Nonetheless, the infinite-dimensional space $\tilHzone$ embeds a weak (in the sense of face moments) $\Hzone$-conformity property.

Let us introduce, locally to any mesh cell $K\in\KK$, the following finite-dimensional flux and potential spaces:
\begin{equation} \label{eq:d.flu.loc}
  \dFlu(K)\defi\left\{\vec{\tau}\in\Hdiv[K]\cap\Hcurlzw[K]\mid\Div\vec{\tau}\in\Pol^{\ell-1}(K),\,\vec{\tau}_{\mid\pK}{\cdot}\vN_{\pK}\in\Pol^{k-1}(\FF_K)\right\},
\end{equation}
\begin{equation} \label{eq:d.pot.loc}
  \dPot(K)\defi\left\{v\in\Hone[K]\mid\Div(\tb\Grad v)\in\Pol^{\ell-1}(K),\,(\tb\Grad v)_{\mid\pK}{\cdot}\vN_{\pK}\in\Pol^{k-1}(\FF_K)\right\},
\end{equation}
where $\Hcurlzw[K]\defi\big\{\vec{\tau}\in\vLL[K]\mid\Curl(\tb^{-1}\vec{\tau})\equiv\vec{0}\big\}$.
The latter space is to be replaced by $\vec{H}(\sRot^0_{\tb^{-1}};K)\defi\big\{\vec{\tau}\in\vLL[K]\mid\sRot(\tb^{-1}\vec{\tau})\equiv 0\big\}$ whenever $d=2$, where $\sRot\vec{z}\defi\Div(\vec{z}^\perp)$.
It is an easy matter to realize that $\dFlu(K)=\tb\Grad\big(\dPot(K)\big)$. Also, for all $v\in \dPot(K)$ and $\vec{\tau}\in\dFlu(K)$, there holds
\begin{equation} \label{eq:ibp}
  \big(\tb^{-1}(\tb\Grad v),\vec{\tau}\big)_K=-\big(\Div\vec{\tau},\projK(v)\big)_K+\big(\vec{\tau}_{\mid\pK}{\cdot}\vN_{\pK},\projFFK(v_{\mid\pK})\big)_{\pK},
\end{equation}
where $\Div\vec{\tau}\in\Pol^{\ell-1}(K)$ satisfies
\begin{equation} \label{eq:div}
  (\Div\vec{\tau},\phi)_K=-\big(\vprojGK(\vec{\tau}),\Grad\phi\big)_K+\big(\vec{\tau}_{\mid\pK}{\cdot}\vN_{\pK},\phi_{\mid\pK}\big)_{\pK}\qquad\forall\phi\in\Pol^{\ell-1}(K).
\end{equation}
The above relations, combined with dimension-counting arguments, yield the bijectivity of the local (flux and potential) DoF maps below:
\begin{itemize}
  \item[$\bullet$] $\dFlu(K)\to\Gol^{\ell-2}(K)\times\Pol^{k-1}(\FF_K);\;\vec{\tau}\mapsto\big(\vprojGK(\vec{\tau}),\vec{\tau}_{\mid\pK}{\cdot}\vN_{\pK}\big)$;
  \item[$\bullet$] $\dPot(K)\to\Pol^{\ell-1}(K)\times\Pol^{k-1}(\FF_K);\;v\mapsto\big(\projK(v),\projFFK(v_{\mid\pK})\big)$.
\end{itemize}
\begin{remark}[Boundedness of the inverse DoF maps] \label{re:bnd.inv}
  Let us begin with the flux space $\dFlu(K)$. Starting from~\eqref{eq:div}, and invoking standard inverse and discrete trace inequalities, it comes
  \begin{equation} \label{eq:bnd.div}
    \|\Div\vec{\tau}\|_K\lesssim h_K^{-1}\|\vprojGK(\vec{\tau})\|_K+h_K^{-\frac12}\|\vec{\tau}_{\mid\pK}{\cdot}\vN_{\pK}\|_{\pK}.
  \end{equation}
  Now, starting from~\eqref{eq:ibp}, first remarking that its right-hand side can be rewritten so that
  \[\big(\tb^{-1}(\tb\Grad v),\vec{\tau}\big)_K=-\big(\Div\vec{\tau},\projK(v-\projK[0](v))\big)_K+\big(\vec{\tau}_{\mid\pK}{\cdot}\vN_{\pK},\projFFK(v_{\mid\pK}-\projK[0](v))\big)_{\pK},\]
  then selecting $v\in\dPot(K)$ such that $\tb\Grad v=\vec{\tau}$, and invoking standard Poincar\'e and continuous trace inequalities, we infer
  \begin{equation} \label{eq:bnd.flu}
    \|\tb^{-\frac12}\vec{\tau}\|_K\lesssim h_K\|\Div\vec{\tau}\|_K+h_K^{\frac12}\|\vec{\tau}_{\mid\pK}{\cdot}\vN_{\pK}\|_{\pK}.
  \end{equation}
  Combining~\eqref{eq:bnd.div} and~\eqref{eq:bnd.flu} finally yields, for any $\vec{\tau}\in\dFlu(K)$,
  \begin{equation} \label{eq:bnd.hdiv}
    \|\vec{\tau}\|_K+h_K\|\Div\vec{\tau}\|_K\lesssim\|\vprojGK(\vec{\tau})\|_K+h_K^{\frac12}\|\vec{\tau}_{\mid\pK}{\cdot}\vN_{\pK}\|_{\pK}.
  \end{equation}
  Let us now turn to the potential space $\dPot(K)$. Starting again from~\eqref{eq:ibp}, but selecting this time $\vec{\tau}\in\dFlu(K)$ such that $\vec{\tau}=\tb\Grad v$, we first get that
  \[\|\tb^{\frac12}\Grad v\|_K^2\leq\|\Div(\tb\Grad v)\|_K\|\projK(v)\|_K+\|(\tb\Grad v)_{\mid\pK}{\cdot}\vN_{\pK}\|_{\pK}\|\projFFK(v_{\mid\pK})\|_{\pK}.\]
  By~\cite[Lem.~4.4]{CiELe:19}, $\|\Div(\tb\Grad v)\|_K\lesssim h_K^{-1}\|\tb^{\frac12}\Grad v\|_K$ and $\|(\tb\Grad v)_{\mid\pK}{\cdot}\vN_{\pK}\|_{\pK}\lesssim h_K^{-\frac12}\|\tb^{\frac12}\Grad v\|_K$. Therefore, it comes
  \begin{equation} \label{eq:bnd.pot}
    \|\tb^{\frac12}\Grad v\|_K\lesssim h_K^{-1}\|\projK(v)\|_K+h_K^{-\frac12}\|\projFFK(v_{\mid\pK})\|_{\pK}.
  \end{equation}
  Remarking that $\|v\|_K\leq\|v-\projK[0](v)\|_K+\|\projK[0](v)\|_K$ and $\projK[0](v)=\projK[0](\projK(v))$, and invoking the Poincar\'e inequality, finally yields, for any $v\in\dPot(K)$,
  \begin{equation} \label{eq:bnd.h1}
    h_K^{-1}\|v\|_K+\|\Grad v\|_K\lesssim h_K^{-1}\|\projK(v)\|_K+h_K^{-\frac12}\|\projFFK(v_{\mid\pK})\|_{\pK}.
  \end{equation}
\end{remark}
As it will become clear in what follows, the identity~\eqref{eq:ibp} is pivotal in proving discrete primal/mixed equivalence.

At the global level, now, from $\dFlu(\KK)\defi\left\{\vec{\tau}_\KK\in\Hdiv[\KK]\mid\vec{\tau}_K\in\dFlu(K)\;\forall K\in\KK\right\}$ (broken flux space), we define
\begin{equation} \label{eq:d.flu.glo}
  \gFlu\defi\dFlu(\KK)\cap\Hdiv.
\end{equation}
In turn, from $\dPot(\KK)\defi\left\{v_\KK\in\Hone[\KK]\mid v_K\in \dPot(K)\;\forall K\in\KK\right\}$ (broken potential space), we let
\begin{equation} \label{eq:d.pot.glo}
  \gPot\defi\dPot(\KK)\cap\tilHzone.
\end{equation}
Next, we introduce the following skeletal trace spaces, whose definitions mimic~\eqref{eq:c.Hmhalf} and~\eqref{eq:c.Hzphalf}:
\begin{align}
  \dHmhalf&\defi\left\{\tau_{\pKK}\defi(\tau_{\pK})_{K\in\KK}\in\prod_{K\in\KK}\Pol^{k-1}(\FF_K)
  \left |\!
  \begin{array}{l}
    \exists\,\vec{\tau}\in\gFlu\text{ s.t.}
    \\
    \tau_{\pK} = \vec{\tau}_{\mid\pK}{\cdot}\vN_{\pK} \;\forall K \in\KK
  \end{array}
  \right .
  \!\!\!\right\},\label{eq:d.Hmhalf}
  \\
  \dHzphalf&\defi\left\{v_{\pKK}\defi(v_{\pK})_{K\in\KK}\in\prod_{K\in\KK}\Pol^{k-1}(\FF_K)
  \left |\!
  \begin{array}{l}
    \exists\,\til{v}_{\KK}\in\gPot\text{ s.t.}
    \\
    v_{\pK} = \projFFK(\til{v}_{K\mid\pK}) \;\forall K \in\KK
  \end{array}
  \right .
  \!\!\!\right\}.\label{eq:d.Hzphalf}
\end{align}
The space $\dHmhalf$ (resp.~$\dHzphalf$) is isomorphic to $\Pol^{k-1}(\FF)$ (resp.~$\Pol^{k-1}(\FFi)$).
Let $\tau_{\pKK}\in\dHmhalf$ and $v_{\pKK}\in\dHzphalf$. On the one hand, since $\vec{\tau}\in\gFlu\subset\Hdiv$ and $\til{v}_\KK\in\gPot\subset\tilHzone$ (so that $\projF(\llbracket v_\KK\rrbracket_F)\equiv 0$ for all $F\in\FF$), there holds
\begin{align*}
  \langle\tau_{\pKK},v_{\pKK}\rangle_{\pKK}&=\sum_{K\in\KK}\sum_{F\in\FF_K}\varepsilon_{K,F}\big(\vec{\tau}_{\mid F}{\cdot}\vN_F,\projF(\til{v}_{K\mid F})\big)_F\\
  &=\sum_{F\in\FF}\big(\vec{\tau}_{\mid F}{\cdot}\vN_F,\projF(\llbracket\til{v}_{\KK}\rrbracket_F)\big)_F=0.
\end{align*}                                                                                                   
On the other hand, since $\vec{\tau}\in\gFlu\subset\dFlu(\KK)$ (so that $\vec{\tau}_{\mid\pK}{\cdot}\vN_{\pK}\in\Pol^{k-1}(\FF_K)$ for all $K\in\KK$), one may remove the projectors $\projF$ and infer that
\begin{equation} \label{eq:d.fund}
  0=\langle \tau_{\pKK},v_{\pKK}\rangle_{\pKK}=\sum_{K\in\KK}\big(\vec{\tau}_{\mid\pK}{\cdot}\vN_{\pK},\til{v}_{K\mid\pK}\big)_{\pK}=\sum_{K\in\KK}\big((\Div\vec{\tau},\til{v}_K)_K+(\vec{\tau},\Grad\til{v}_K)_K\big).
\end{equation}
The relation~\eqref{eq:d.fund} is a discrete counterpart of~\eqref{eq:c.fund}. As a by-product, the space characterizations below, which are discrete versions of~\eqref{eq:c.Hdiv} and~\eqref{eq:c.Hzone}, hold true:
\begin{align}
  \gFlu&=\big\{\vec{\tau}_\KK\in\dFlu(\KK)\mid\langle\vec{\tau}_{\KK\mid\pKK}{\cdot}\vN_{\pKK},v_{\pKK}\rangle_{\pKK}=0\;\;\forall v_{\pKK}\in\dHzphalf\big\},\label{eq:d.Hdiv}\\
  \gPot&=\big\{v_\KK\in \dPot(\KK)\mid\langle\tau_{\pKK},v_{\KK\mid\pKK}\rangle_{\pKK}=0\;\;\forall\tau_{\pKK}\in\dHmhalf\big\}.\label{eq:d.Hzone}
\end{align}

\subsection{Discrete problems}

We now have at our disposal all the necessary ingredients to write down finite-dimensional counterparts of Problems~\ref{pb:c.primal} to~\ref{pb:c.hmixed}. The translation from the infinite-dimensional setting to the finite-dimensional one is performed based on the Rosetta stone displayed in Table~\ref{ta:equiv}.
\begin{table}[h!]
  \centering
  \begin{tabular}{c|c}
    Infinite-dimensional space & Finite-dimensional twin \\
    \hline\hline
    $\LL$ & $\Pol^{\ell-1}(\KK)$\\
    \hline
    $\Hdiv[\KK]$ & $\dFlu(\KK)$ (cf.~\eqref{eq:d.flu.loc} for local definition)\\
    \hline
    $\Hdiv$ & $\gFlu$ (cf.~\eqref{eq:d.flu.glo})\\
    \hline
    $\Hmhalf$ (cf.~\eqref{eq:c.Hmhalf}) & $\dHmhalf\cong\Pol^{k-1}(\FF)$ (cf.~\eqref{eq:d.Hmhalf})\\
    \hline
    $\Hone[\KK]$ & $\dPot(\KK)$ (cf.~\eqref{eq:d.pot.loc} for local definition)\\
    \hline
    $\Hzone$ & $\gPot$ (cf.~\eqref{eq:d.pot.glo})\\
    \hline
    $\Hzphalf$ (cf.~\eqref{eq:c.Hzphalf}) & $\dHzphalf\cong\Pol^{k-1}(\FFi)$ (cf.~\eqref{eq:d.Hzphalf})\\
    \hline
  \end{tabular}
  \caption{The Rosetta stone.\label{ta:equiv}}
\end{table}

\noindent
The new problems are obtained (i) replacing the infinite-dimensional spaces by their finite-dimensional twins (cf.~Table~\ref{ta:equiv}) and (ii) replacing, within primal formulations, the load $f$ by its cell-wise polynomial projection $\projKK(f)$.
\begin{problem}[Primal] \label{pb:d.primal}
  Find $\til{u}_{\KK}\in\gPot$ s.t.
  \begin{equation} \label{eq:d.primal}
    (\tb\Grad_{\KK}\til{u}_{\KK},\Grad_{\KK}\til{v}_{\KK})_{\Omega}=(f,\projKK(\til{v}_{\KK}))_{\Omega}\qquad\forall\til{v}_{\KK}\in\gPot.
  \end{equation}
\end{problem}
\begin{problem}[Mixed] \label{pb:d.mixed}
  Find $(\vec{\sigma},u)\in\gFlu\times\Pol^{\ell-1}(\KK)$ s.t.
  \begin{subequations} \label{eq:d.mixed}
    \begin{alignat}{2}
      (\tb^{-1}\vec{\sigma},\vec{\tau})_{\Omega}+(u,\Div\vec{\tau})_{\Omega} &= 0 &\qquad&\forall\vec{\tau}\in\gFlu, \label{eq:d.mixed.a}
      \\
      -(\Div\vec{\sigma},v)_{\Omega} &= (f,v)_{\Omega} &\qquad&\forall v\in\Pol^{\ell-1}(\KK). \label{eq:d.mixed.b}
    \end{alignat}
  \end{subequations}
\end{problem}
\begin{problem}[Hybridized primal] \label{pb:d.hprimal}
  Find $(u_\KK,\sigma_{\pKK})\in\dPot(\KK)\times\dHmhalf$ s.t.
  \begin{subequations} \label{eq:d.hprimal}
    \begin{alignat}{2}
      (\tb\Grad_\KK u_\KK,\Grad_\KK v_\KK)_{\Omega}-\langle\sigma_{\pKK},v_{\KK\mid\pKK}\rangle_{\pKK} &= (f,\projKK(v_\KK))_{\Omega} &\qquad&\forall v_\KK\in\dPot(\KK), \label{eq:d.hprimal.a}
      \\
      \langle\tau_{\pKK},u_{\KK\mid\pKK}\rangle_{\pKK} &= 0 &\qquad&\forall\tau_{\pKK}\in\dHmhalf. \label{eq:d.hprimal.b}
    \end{alignat}
  \end{subequations}
\end{problem}
\begin{problem}[Hybridized mixed] \label{pb:d.hmixed}
  Find $(\vec{\sigma}_\KK,u_\KK,u_{\pKK})\in\dFlu(\KK)\times\Pol^{\ell-1}(\KK)\times\dHzphalf$ s.t.
  \begin{subequations} \label{eq:d.hmixed}
    \begin{alignat}{2}
      (\tb^{-1}\vec{\sigma}_\KK,\vec{\tau}_\KK)_{\Omega}+(u_\KK,\Div_\KK\vec{\tau}_\KK)_{\Omega}-\langle\vec{\tau}_{\KK\mid\pKK}{\cdot}\vN_{\pKK}, u_{\pKK}\rangle_{\pKK} &= 0 &\quad&\forall\vec{\tau}_\KK\in\dFlu(\KK), \label{eq:d.hmixed.a}
      \\
      -(\Div_\KK\vec{\sigma}_\KK,v_\KK)_{\Omega} &= (f,v_\KK)_{\Omega} &\quad&\forall v_\KK\in\Pol^{\ell-1}(\KK), \label{eq:d.hmixed.b}
      \\
      \langle\vec{\sigma}_{\KK\mid\pKK}{\cdot}\vN_{\pKK}, v_{\pKK}\rangle_{\pKK} &= 0&\quad&\forall v_{\pKK}\in\dHzphalf. \label{eq:d.hmixed.c}
    \end{alignat}
  \end{subequations}
\end{problem}
\noindent
Remark that Problems~\ref{pb:d.mixed} and~\ref{pb:d.hprimal} are conforming approximations of Problems~\ref{pb:c.mixed} and~\ref{pb:c.hprimal}, respectively. On the contrary, Problems~\ref{pb:d.primal} and~\ref{pb:d.hmixed} are nonconforming approximations of Problems~\ref{pb:c.primal} and~\ref{pb:c.hmixed}, respectively.
\begin{remark}[Well-posedness] \label{re:fdc.wp}
  All four discrete problems above are (uniformly) well-posed. The well-posedness of the primal Problem~\ref{pb:d.primal} is standard; we refer, e.g., to~\cite[Lem.~3.1]{Lemai:21}. For the hybridized primal Problem~\ref{pb:d.hprimal}, we refer the reader to~\cite[Thm.~2]{RaTho:77a}. One only needs to revisit, therein, the proof of the LBB condition for the coupling (Eq.~(6.29)), which can be done leveraging the stability estimate~\eqref{eq:bnd.h1}. For the mixed Problem~\ref{pb:d.mixed}, in turn, well-posedness can be established following the lines of~\cite{RaTho:77b}. More particularly, one needs to revisit, therein, the proofs of the local result of Lem.~4 and of the global result of Lem.~5 (which, in turn, is based on~\cite[Eq.~(6.29)]{RaTho:77a}). This can be achieved leveraging the stability estimates~\eqref{eq:bnd.flu} and~\eqref{eq:bnd.h1}, respectively. Finally, the well-posedness of the hybridized mixed Problem~\ref{pb:d.hmixed} hinges on the well-posedness of the mixed problem, and on some LBB condition for the coupling, which can be inferred based on the stability estimate~\eqref{eq:bnd.hdiv}.
\end{remark}
\noindent
Proceeding exactly as in the infinite-dimensional setting (cf.~Propositions~\ref{pr:c.eqprimal} and~\ref{pr:c.eqmixed}), but invoking this time the space characterizations~\eqref{eq:d.Hzone} and~\eqref{eq:d.Hdiv}, equivalence can be established between the two primal formulations (Problems~\ref{pb:d.primal} and~\ref{pb:d.hprimal}) on the one hand, and between the two mixed formulations (Problems~\ref{pb:d.mixed} and~\ref{pb:d.hmixed}) on the other hand.

\subsection{Primal/mixed equivalence}

Let us now investigate in more depth the connections between primal and mixed formulations. Below, we establish finite-dimensional counterparts of Lemmas~\ref{le:c.ptom} and~\ref{le:c.mtop}, proving that a discrete equivalent of the (infinite-dimensional) diagram of Figure~\ref{fig:c.diag} is valid.
\begin{lemma}[From primal to mixed] \label{le:d.ptom}
  Assume that the couple $(u_\KK,\sigma_{\pKK})\in\dPot(\KK)\times\dHmhalf$ solves Problem~\ref{pb:d.hprimal}. Then, there holds $\tb\Grad_{\KK}u_\KK\in\gFlu$, and the couple $(\vec{\sigma}\defi\tb\Grad_{\KK}u_\KK,u\defi\projKK(u_{\KK}))\in\gFlu\times\Pol^{\ell-1}(\KK)$ solves Problem~\ref{pb:d.mixed}.
\end{lemma}
\begin{proof}
  The result falls in two steps.\\
  (a) First, let us rewrite Eq.~\eqref{eq:d.hprimal.a} cell by cell: inside $K\in\KK$, letting $\vec{\sigma}_K\defi\tb\Grad u_K$, and using the fact that $\sigma_{\pK}\in\Pol^{k-1}(\FF_K)$, it comes, for all $v_K\in\dPot(K)$,
  \[\big(\tb^{-1}(\tb\Grad v_K),\vec{\sigma}_K\big)_K-\big(\sigma_{\pK},\projFFK(v_{K\mid\pK})\big)_{\pK}=\big(f,\projK(v_K)\big)_K.\]
  Since $\vec{\sigma}_K\in\dFlu(K)$ (recall that $u_K\in\dPot(K)$), an application of formula~\eqref{eq:ibp} to the term $\big(\tb^{-1}(\tb\Grad v_K),\vec{\sigma}_K\big)_K$ then yields
  \[-\big(\Div\vec{\sigma}_K,\projK(v_K)\big)_K+\big(\vec{\sigma}_{K\mid\pK}{\cdot}\vN_{\pK}-\sigma_{\pK},\projFFK(v_{K\mid\pK})\big)_{\pK}=\big(\projK(f_{\mid K}),\projK(v_K)\big)_K.\]
  The above relation is valid for all $v_K\in\dPot(K)$. Consequently, recalling that the potential DoF map $v\mapsto\big(\projK(v),\projFFK(v_{\mid\pK})\big)$ is bijective from $\dPot(K)$ to $\Pol^{\ell-1}(K)\times\Pol^{k-1}(\FF_K)$, and that there holds $\vec{\sigma}_K\in\dFlu(K)$ (so that $\Div\vec{\sigma}_K\in\Pol^{\ell-1}(K)$ and $\vec{\sigma}_{K\mid\pK}{\cdot}\vN_{\pK}\in\Pol^{k-1}(\FF_K)$) and $\sigma_{\pK}\in\Pol^{k-1}(\FF_K)$, we both infer that $-\Div\vec{\sigma}_K=\projK(f_{\mid K})$ a.e.~in $K$, and that $\vec{\sigma}_{K\mid\pK}{\cdot}\vN_{\pK}=\sigma_{\pK}$ a.e.~on $\pK$. Then, given that $\sigma_{\pKK}\in\dHmhalf$, it follows that $\vec{\sigma}_\KK(=\tb\Grad_\KK u_\KK)\in\gFlu\subset\Hdiv$, and that $-(\Div\vec{\sigma}_\KK,v)_{\Omega}=(f,v)_{\Omega}$ for all $v\in\Pol^{\ell-1}(\KK)$.\\
  (b) Second, we prove that $(\tb^{-1}\vec{\sigma}_\KK,\vec{\tau})_{\Omega}+(\projKK(u_\KK),\Div\vec{\tau})_{\Omega}=0$ for all $\vec{\tau}\in\gFlu$. The conclusion then follows setting $(\vec{\sigma}\defi\vec{\sigma}_\KK,u\defi\projKK(u_\KK))\in\gFlu\times\Pol^{\ell-1}(\KK)$. For any $\vec{\tau}\in\gFlu$, letting $\tau_{\pKK}\in\prod_{K\in\KK}\Pol^{k-1}(\FF_K)$ be such that $\tau_{\pKK}\defi\vec{\tau}_{\mid\pKK}{\cdot}\vN_{\pKK}$, there holds $\tau_{\pKK}\in\dHmhalf$. Then, recalling that $\vec{\sigma}_K=\tb\Grad u_K$ for all $K\in\KK$, and invoking~\eqref{eq:d.hprimal.b}, it comes
  \begin{align*}
    (\tb^{-1}\vec{\sigma}_\KK,\vec{\tau})_{\Omega}+(\projKK(u_\KK),\Div\vec{\tau})_{\Omega}&=\sum_{K\in\KK}\big((\Grad u_K,\vec{\tau})_K+(u_K,\Div\vec{\tau})_K\big)\\
    &=\sum_{K\in\KK}(\tau_{\pK},u_{K\mid\pK})_{\pK}=\langle\tau_{\pKK},u_{\KK\mid\pKK}\rangle_{\pKK}=0.
  \end{align*}
\end{proof}
\begin{lemma}[From mixed to primal] \label{le:d.mtop}
  Assume that the triple $(\vec{\sigma}_\KK,u_\KK,u_{\pKK})\in\dFlu(\KK)\times\Pol^{\ell-1}(\KK)\times\dHzphalf$ solves Problem~\ref{pb:d.hmixed}. Then, letting $\til{u}_{\KK}$ be the unique element of $\gPot$ such that, for all $K\in\KK$, $\projK(\til{u}_K)=u_K$ and $\projFFK(\til{u}_{K\mid\pK})=u_{\pK}$, $\til{u}_\KK\in\gPot$ solves Problem~\ref{pb:d.primal}.
\end{lemma}
\begin{proof}
  The result also falls in two steps.\\
  (a) First, let us rewrite Eq.~\eqref{eq:d.hmixed.a} cell by cell: inside $K\in\KK$, replacing $u_K$ and $u_{\pK}$ by their expressions in terms of $\til{u}_K\in\dPot(K)$, we infer that, for all $\vec{\tau}_K\in\dFlu(K)$,
  \[\big(\tb^{-1}\vec{\sigma}_K,\vec{\tau}_K\big)_K+\big(\projK(\til{u}_K),\Div\vec{\tau}_K\big)_K-\big(\vec{\tau}_{K\mid\pK}{\cdot}\vN_{\pK},\projFFK(\til{u}_{K\mid\pK})\big)_{\pK} = 0.\]
  A simple application of~\eqref{eq:ibp} then implies that $\vec{\sigma}_K=\tb\Grad\til{u}_K(\in\dFlu(K))$ a.e.~in $K$.\\
  (b) Second, we prove that $(\tb\Grad_\KK\til{u}_\KK,\Grad_{\KK}\til{v}_{\KK})_{\Omega}=(f,\projKK(\til{v}_\KK))_{\Omega}$ for all $\til{v}_\KK\in\gPot$, i.e., that $\til{u}_\KK\in\gPot$ solves Problem~\ref{pb:d.primal}, which shall conclude the proof. For any $\til{v}_{\KK}\in\gPot$, letting $(v_{\KK},v_{\pKK})\in\Pol^{\ell-1}(\KK)\times\prod_{K\in\KK}\Pol^{k-1}(\FF_K)$ be such that $v_K\defi\projK(\til{v}_K)$ and $v_{\pK}\defi\projFFK(\til{v}_{K\mid\pK})$ for all $K\in\KK$, there holds $(v_{\KK},v_{\pKK})\in\Pol^{\ell-1}(\KK)\times\dHzphalf$. Then, recalling that $\vec{\sigma}_K=\tb\Grad\til{u}_K$ for all $K\in\KK$, and invoking again~\eqref{eq:ibp} together with~\eqref{eq:d.hmixed.b} and~\eqref{eq:d.hmixed.c}, we infer
  \begin{align*}
    (\tb\Grad_\KK\til{u}_\KK,\Grad_{\KK}\til{v}_{\KK})_{\Omega}&=\sum_{K\in\KK}\big(\tb^{-1}(\tb\Grad\til{v}_K),\vec{\sigma}_K\big)_K\\
    &=\sum_{K\in\KK}\left(-(\Div\vec{\sigma}_K,v_K)_K+(\vec{\sigma}_{K\mid\pK}{\cdot}\vN_{\pK},v_{\pK})_{\pK}\right)\\
    &=-(\Div_\KK\vec{\sigma}_\KK,v_\KK)_{\Omega}+\langle\vec{\sigma}_{\KK\mid\pKK}{\cdot}\vN_{\pKK}, v_{\pKK}\rangle_{\pKK}\\
    &=(f,v_\KK)_{\Omega}=(f,\projKK(\til{v}_\KK))_{\Omega}.
  \end{align*}
\end{proof}

In this section, we have thus been able to exhibit a finite-dimensional construction for which primal/mixed equivalence does hold true. This construction, because $\dFlu(\KK)=\tb\Grad_{\KK}\big(\dPot(\KK)\big)$, exactly preserves the structure of the constitutive relation from the model at hand. It also preserves (full) $\Hdiv$-conformity for the flux and weak (in the sense of face moments) $\Hzone$-conformity for the potential. However, because the spaces $\dPot(K)$ are defined implicitly (via the solutions to PDEs), they are not analytically known in general (and, even if they are, they may contain non-polynomial functions), which essentially prevents the practical implementation of the corresponding one-level methods.
Nevertheless, these spaces are at the very core of a number of oversampling-free multiscale (a.k.a.~two-level) approaches. In the primal context, first, one can cite the MsFE \`a la Crouzeix--Raviart methods of~\cite{LBLLo:13} (case $\ell=0$ and $k=1$, not considered herein) and~\cite{LBLLo:14} (case $\ell=k=1$, therein studied in a perforated setting), and their generalization to arbitrary polynomial degrees of~\cite{CiELe:19} (under the name of MsHHO method). A small, yet important (as it might condition primal/mixed equivalence) difference between~\cite{LBLLo:13,LBLLo:14} and~\cite{CiELe:19} lies in the discretization of the right-hand side of the problem (which is $(f,\projKK(\til{v}_{\KK}))_{\Omega}$ for MsHHO, in place of $(f,\til{v}_{\KK})_{\Omega}$ for MsFE \`a la Crouzeix--Raviart). In the hybridized primal context, in turn, one can cite the MHM method of~\cite{HPVal:13,AHPVa:13,PaVaV:17}, whose first-level equivalence with the MsHHO method has already been established in~\cite{CFELV:22} (therein, the corresponding version of the MHM method is referred to as ``fully explicit''). In the mixed context, now, our construction generalizes to polytopal cells and arbitrary polynomial degrees the method of~\cite{ChenH:03}. Finally, in the hybridized mixed setting, the first-level method from Problem~\ref{pb:d.hmixed} does not seem to have ever appeared in the literature (note that~\cite{ELShi:15} is actually based on local Dirichlet problems).

\section{Projection methods} \label{se:proj}

We aim to explore here situations for which the construction from Section~\ref{se:fdc} would be directly usable as a one-level method. We shall assume, throughout this section, that the (symmetric) mobility tensor field $\tb$ is cell-wise constant over $\KK$, i.e.,
\begin{equation} \label{eq:diff.cwc}
  \tb_K\defi\tb_{\mid K}\in\R^{d\times d}_{\rm sym}\qquad\forall K\in\KK.
\end{equation}

\subsection{A silver bullet}

Let us temporarily assume (in this subsection only) that the partition $\KK$ of $\Omega$ is simplicial. 
Recalling definition~\eqref{eq:d.pot.loc}, and given assumption~\eqref{eq:diff.cwc}, we begin by remarking that, for any simplex $K\in\KK$, $\dPot[0,1](K)=\Pol^1(K)$ (of dimension $d{+}1$). At the global level, it holds $\gPot[0,1]=\Pol^1(\KK)\cap\tilHzone[1]$, i.e., $\gPot[0,1]=\CR^1_0(\Omega)$, where $\CR^1(\Omega)$ is the so-called Crouzeix--Raviart space, originally introduced in~\cite{CroRa:73} (in its vector-valued fashion) in combination with $\Pol^0(\KK)$ as a stable pair for the Stokes problem. Because we exclude the case $\ell=0$ from our study, let us henceforth concentrate on the space $\dPot[1,1](K)$. Letting $\psi_K(\vec{x})\defi\frac{1}{2}\tb_K^{-1}(\vec{x}-\bar{\vec{x}}_K){\cdot}(\vec{x}-\bar{\vec{x}}_K)$, one can easily verify that
\begin{equation} \label{eq:d.pot.loc:sim.oo}
  \dPot[1,1](K)=\Pol^1(K)\oplus\Pol^0(K)\psi_K
\end{equation}
(of dimension $d{+}2$). Notice, indeed, that $\Div(\tb_K\Grad\psi_K)\equiv d\in\Pol^0(K)$ and that, for any $F\in\FF_K$,
\[[\tb_K\Grad\psi_K]_{\mid F}(\vec{x}){\cdot}\vN_{K,F}=(\vec{x}-\bar{\vec{x}}_K){\cdot}\vN_{K,F}=\mathrm{d}_{K,F}\qquad\text{for all $\vec{x}\in F$},\]
where $\mathrm{d}_{K,F}\in\Pol^0(F)$ denotes the (orthogonal) distance between $\bar{\vec{x}}_K$ and the affine hyperplane containing $F$. Observe that $\Pol^1(K)\subset\dPot[1,1](K)\subset\Pol^2(K)$. At the global level, $\gPot[1,1]=\dPot[1,1](\KK)\cap\tilHzone[1]$ can be viewed as a bubble-enriched (homogeneous) Crouzeix--Raviart space.
Note that the function $\psi_K$ is not, strictly speaking, a bubble, but the space $\dPot[1,1](K)$ may equivalently be written as the 
direct sum (i) of $\spa\{(\psi_K^{\partial,F})_{F\in\FF_K}\}(=\Pol^1(K))$ such that, for all $F\in\FF_K$, $\Div(\tb_K\Grad\psi_K^{\partial,F})\equiv 0$ and $\pi_{F'}^0(\psi_{K\mid F'}^{\partial,F})=\delta_{FF'}$ for all $F'\in\FF_K$, and (ii) of $\spa\{\psi_K^\circ\}$ such that $\Div(\tb_K\Grad\psi_K^\circ)\equiv 1$ and $\pi_{F'}^0(\psi_{K\mid F'}^\circ)=0$ for all $F'\in\FF_K$.
The (quadratic) function $\psi_K^\circ$ is precisely what might be called a (weak) bubble.
Now, following the construction from Section~\ref{se:fdc}, the local flux space $\dFlu[1,1](K)$ is given by $\dFlu[1,1](K)=\tb_K\Grad(\dPot[1,1](K))$. From~\eqref{eq:d.pot.loc:sim.oo}, it is straightforward to see that
\begin{equation} \label{eq:d.flu.loc:sim.oo}
  \dFlu[1,1](K)=\tb_K\Gol^0(K)\overset{\tb_K^{-1}\perp}{\oplus}\Pol^0(K)(\vec{x}-\bar{\vec{x}}_K)
\end{equation}
(of dimension $d{+}1$), where ``$\tb_K^{-1}\perp$'' stands for $\tb_K^{-1}$-weighted $\vLL[K]$-orthogonality.
Because $\tb_K\Gol^0(K)=\vPol^0(K)$, we infer that $\dFlu[1,1](K)=\RT^1(K)$, where $\RT^1(K)$ is the local lowest-order Raviart--Thomas space (cf.~\eqref{eq:RTN}).
Observe that $\dFlu[1,1](K)$ satisfies $\vPol^0(K)\subset\dFlu[1,1](K)\subset\vPol^1(K)$.
At the global level,
\[\gFlu[1,1]=\RT^1(\KK)\cap\Hdiv,\]
that is $\gFlu[1,1]=\RT^1(\Omega)$, where $\RT^1(\Omega)$ is the global lowest-order Raviart--Thomas space, originally introduced (for arbitrary degrees) in~\cite{RaTho:77b} in combination with $\Pol^0(\KK)$ as a stable pair for the mixed Poisson problem in 2D (then extended to 3D by N\'ed\'elec in~\cite{Nedel:80}).

For such a choice of discrete spaces, the mixed formulation of Problem~\ref{pb:d.mixed} precisely coincides with the lowest-order Raviart--Thomas scheme from~\cite{RaTho:77b}. By discrete primal/mixed equivalence, the latter problem is equivalent to Problem~\ref{pb:d.primal}, that is: find $\til{u}_\KK\in\gPot[1,1]$ such that
\begin{equation} \label{eq:d.primal:sim.oo}
  (\tb\Grad_{\KK}\til{u}_{\KK},\Grad_{\KK}\til{v}_{\KK})_{\Omega}=(f,\projKK[0](\til{v}_{\KK}))_{\Omega}\qquad\forall\til{v}_{\KK}\in\gPot[1,1].
\end{equation}
The lowest-order (mixed) Raviart--Thomas method on simplices is hence equivalent to a bubble-enriched version of the (primal) Crouzeix--Raviart method, modulo a slight modification of its right-hand side.
Quite remarkably, the solution $\til{u}_\KK\in\gPot[1,1]\subset\tilHzone[1]$ to Problem~\eqref{eq:d.primal:sim.oo} also satisfies that $\tb\Grad_\KK\til{u}_\KK\in\RT^1(\Omega)\subset\Hdiv$.
This result first appeared in the literature under that form in~\cite[Sec.~4]{ChenZ:93} for triangular meshes, and was then generalized to tetrahedral meshes in~\cite[Sec.~10]{ArbCh:95}. However, both statements assume that all the simplices of the mesh are regular (equilateral, in 2D). It was indeed not realized, at the time, that the result is actually valid for any simplical mesh, which we make clear in this work. Whereas it is evident that $\tb_K\Grad\psi_K\in\RT^1(K)$, in~\cite{ChenZ:93,ArbCh:95}, the space $\Pol^1(K)$ was instead enriched with the span of $\psi_K^\circ$, and it was not noticed that $\tb_K\Grad\psi_K^\circ\in\RT^1(K)$ whatever the shape of the simplex $K$ (in fact, there holds $\psi_K^\circ=\alpha\psi_K+\beta$ for some $\alpha,\beta\in\R$, $\alpha\neq 0$).

\begin{remark}[The case of hyperrectangles]
  A similar equivalence result is valid on hyperrectangles (i.e., rectangles in 2D, rectangular cuboids in 3D) for diagonal (still cell-wise constant) mobility tensor fields.
  Let us examplify it in 2D.
  For any rectangle $K\in\KK$, letting
  \[\phi_K(\vec{x})\defi\frac{1}{2}\tb_K^{-1}\begin{pmatrix}(x-\bar{x}_K)\\-(y-\bar{y}_K)\end{pmatrix}{\cdot}\begin{pmatrix}(x-\bar{x}_K)\\(y-\bar{y}_K)\end{pmatrix},\]
  one can easily verify that $\dPot[0,1](K)=\Pol^1(K)\oplus\Pol^0(K)\phi_K$ (of dimension $4$). At the global level, there holds $\gPot[0,1]=\til{\mathcal{RT}}^1_0(\Omega)$, where $\til{\mathcal{RT}}^1(\Omega)$ (not to be confused with $\RT^1(\Omega)$) is the Rannacher--Turek space, introduced in~\cite{RaTur:92} (in its vector-valued fashion) in combination with $\Pol^0(\KK)$ as a stable pair for the Stokes problem. Turning to the case $\ell=1$, one can check that $\dPot[1,1](K)=\dPot[0,1](K)\oplus\Pol^0(K)\psi_K$ (of dimension $5$), where we have enriched the local space with the same function as in the simplicial setting. Also in that case, we have $\Pol^1(K)\subset\dPot[1,1](K)\subset\Pol^2(K)$. There holds
  \[\vPol^0(K)\subset\dFlu[1,1](K)=\tb_K\Grad\big(\dPot[1,1](K)\big)=\begin{pmatrix}\spa\{1,(x-\bar{x}_K)\}\\\spa\{1,(y-\bar{y}_K)\}\end{pmatrix}\subset\vPol^1(K)\]
  (of dimension $4$). The local (resp.~global) flux space $\dFlu[1,1](K)$ (resp.~$\gFlu[1,1]$) is nothing but the local (resp.~global) lowest-order Raviart--Thomas space on rectangles. We recover here a known result (cf.~\cite[Secs.~6 and 9]{ArbCh:95}): the lowest-order (mixed) Raviart--Thomas method on hyperrectangles is equivalent to a bubble-enriched version of the (primal) Rannacher--Turek method, modulo a slight modification of its right-hand side.
\end{remark}

\begin{remark}[Static condensation]
  We focus on Problem~\eqref{eq:d.primal:sim.oo} (similar considerations remain valid in the hyperrectangular case; see~\cite[Sec.~6]{ArbCh:95} in 2D).
  To any $\til{v}_\KK\in\gPot[1,1]$, one may associate the function $\til{v}^{\mathcal{CR}}_\KK\in\CR^1_0(\Omega)$ s.t.~$\projF[0](\til{v}^{\mathcal{CR}}_{\KK\mid F})=\projF[0](\til{v}_{\KK\mid F})$ for all $F\in\FF$.
  Based on the latter definition, and because $\til{v}^{\mathcal{CR}}_\KK\in\Pol^1(\KK)$, one can prove that, for all $K\in\KK$,
  \begin{align} 
    \Grad\til{v}^{\mathcal{CR}}_K&=\frac{1}{|K|}\sum_{F\in\FF_K}|F|\projF[0](\til{v}_{K\mid F})\vN_{K,F},\label{eq:CR.grad}\\
    \projK[0](\til{v}^{\mathcal{CR}}_K)&=\frac{1}{|K|}\sum_{F\in\FF_K}\frac{|F|\mathrm{d}_{K,F}}{d}\projF[0](\til{v}_{K\mid F}).\label{eq:CR.mean}
  \end{align}
  Locally to any $K\in\KK$, one can then compute $\tb_K\Grad\til{v}_K\in\dFlu[1,1](K)$ by combining the formula~\eqref{eq:ibp} with the $\tb_K^{-1}$-weighted $\vLL[K]$-orthogonal decomposition~\eqref{eq:d.flu.loc:sim.oo} of $\dFlu[1,1](K)$. Letting $\eta_K\defi\int_K|\tb_K^{-\nicefrac12}(\vec{x}-\bar{\vec{x}}_K)|^2$, a tedious yet straightforward calculation yields
  \begin{equation} \label{eq:RT.flux}
    \tb_K\Grad\til{v}_K=\tb_K\Grad\til{v}_K^{\mathcal{CR}}-\left[\frac{d|K|}{\eta_K}\big(\projK[0](\til{v}_K)-\projK[0](\til{v}^{\mathcal{CR}}_K)\big)\right](\vec{x}-\bar{\vec{x}}_K),
  \end{equation}
  where we also used the formulas~\eqref{eq:CR.grad} and~\eqref{eq:CR.mean}.
  Now, testing Problem~\eqref{eq:d.primal:sim.oo} with $\til{v}_\KK\in\gPot[1,1]$ such that $\projF[0](\til{v}_{\KK\mid F})=0$ for all $F\in\FF$ (so that $\til{v}_\KK^{\mathcal{CR}}\equiv 0$) and $\pi^0_{K'}(\til{v}_{K'})=\delta_{KK'}$ for all $K'\in\KK$, and leveraging the expression~\eqref{eq:RT.flux} for both the trial $\til{u}_\KK$ and test $\til{v}_\KK$ functions from~\eqref{eq:d.primal:sim.oo} (along with the $\tb_K^{-1}$-weighted $\vLL[K]$-orthogonality of the local decomposition~\eqref{eq:d.flu.loc:sim.oo}), we infer
  \begin{equation*}
    \frac{d^2|K|^2}{\eta_K}\left(\projK[0](\til{u}_K)-\projK[0](\til{u}^{\mathcal{CR}}_K)\right)=|K|\projK[0](f_{\mid K}),
  \end{equation*}
  which, in turn, yields
  \begin{equation} \label{eq:sc.mean}
    \projK[0](\til{u}_K)=\projK[0](\til{u}^{\mathcal{CR}}_K)+\frac{\eta_K}{d^2|K|}\projK[0](f_{\mid K}).
  \end{equation}
  Combining~\eqref{eq:RT.flux} (applied to $\til{u}_K$) and~\eqref{eq:sc.mean}, we thus obtain that
  \begin{equation} \label{eq:sc.flux}
    \tb_K\Grad\til{u}_K=\tb_K\Grad\til{u}_K^{\mathcal{CR}}-\frac{\projK[0](f_{\mid K})}{d}(\vec{x}-\bar{\vec{x}}_K).
  \end{equation}
  Besides, testing Problem~\eqref{eq:d.primal:sim.oo} with $\til{v}^{\mathcal{CR}}_{\KK}\in\CR^1_0(\Omega)\subset\gPot[1,1]$ and invoking, once again, the $\tb_K^{-1}$-weighted $\vLL[K]$-orthogonality of the local decomposition~\eqref{eq:d.flu.loc:sim.oo}, it is easily seen that $\til{u}^{\mathcal{CR}}_\KK\in\CR^1_0(\Omega)$ solves
  \begin{equation} \label{eq:sc.syst}
    (\tb\Grad_{\KK}\til{u}^{\mathcal{CR}}_{\KK},\Grad_{\KK}\til{v}^{\mathcal{CR}}_{\KK})_{\Omega}=(f,\projKK[0](\til{v}^{\mathcal{CR}}_{\KK}))_{\Omega}\qquad\forall\til{v}^{\mathcal{CR}}_{\KK}\in\CR^1_0(\Omega).
  \end{equation}
  Let us recap. We have thus been able to trade the initial saddle-point system corresponding to the (mixed) Raviart--Thomas problem for~\eqref{eq:sc.syst}, which is a nice SPD system with reduced size.
  The solution strategy then goes as follows: (i) one solves the condensed SPD system~\eqref{eq:sc.syst} for $\til{u}^{\mathcal{CR}}_\KK$ then, (ii) based on the formulas~\eqref{eq:sc.mean} and~\eqref{eq:sc.flux}, one locally reconstructs the solution $\til{u}_\KK$ to~\eqref{eq:d.primal:sim.oo}.
  It is easily seen that, for any $K\in\KK$,
  \[\til{u}_K=\til{u}_K^{\mathcal{CR}}-\frac{\projK[0](f_{\mid K})}{d}(\psi_K-c_K)\qquad\text{with}\qquad c_K=\frac{(2+d)\eta_K}{2d|K|}.\]
  The link between the lowest-order Raviart--Thomas and Crouzeix--Raviart methods, as well as the resulting solution strategy for mixed methods, has first been reported in 2D in~\cite{ArBre:85,Marin:85} for cubic (strong) bubble enrichment (see Example~\ref{ex:ArBre} for more details). It has been extended to quadratic (weak) bubble enrichment and to 3D in~\cite[Sec.~10]{ArbCh:95} (based on the preliminary work~\cite{ChenZ:93}); cf.~also~\cite{VohWo:13}.
\end{remark}

We have been investigating so far situations for which $\dFlu(\KK)=\tb\Grad_\KK\big(\dPot(\KK)\big)$, so that the constitutive relation from the problem at hand is exactly preserved. A necessary condition for a discrete flux space to be admissible in the above sense is that, for all $K\in\KK$, $\Curl\big(\tb_K^{-1}\dFlu(K)\big)=\{\vec{0}\}$ (in 2D, $\sRot\big(\tb_K^{-1}\dFlu(K)\big)=\{0\}$). Let us now assume that the flux space is cell-wise polynomial. For a cell-wise constant mobility field, the curl-free condition imposes that $\dFlu(K)=\tb_K\Grad\big(\mathcal{Q}^n(K)\big)$ for all $K\in\KK$, where $\mathcal{Q}^n(K)$ is some scalar-valued polynomial space of total degree $n\in\N^\star$. Whereas this structural condition is satisfied by the lowest-order Raviart--Thomas flux space on simplices or hyperrectangles (for diagonal mobility), we see that it is violated as soon as $\tb_K\kGol^1(K)\subseteq\dFlu(K)$, which is, in particular, verified as soon as $\vPol^1(K)\subseteq\dFlu(K)$. Therefore, for higher-order Raviart--Thomas spaces, or for Brezzi--Douglas--Marini spaces (of any order), for instance, this condition will necessarily be violated. We hence see that the lowest-order Raviart--Thomas space constitutes, in this respect, some kind of silver bullet. We recall, in passing, that Brezzi--Douglas--Marini spaces were introduced in~\cite{BrDoM:85} for the mixed discretization of the Poisson problem in 2D (then extended to 3D by N\'ed\'elec in~\cite{Nedel:86}). Now, a legitimate question is the following: what kind of primal equivalent rewriting may one obtain for mixed methods based on local (polynomial) flux spaces violating the condition $\Curl\big(\tb_K^{-1}\dFlu(K)\big)=\{\vec{0}\}$?
The answer to this question was first discussed in~\cite{ArbCh:95} (based on the seminal work~\cite{ArBre:85}), where the concept of projection finite element method was introduced.

\subsection{The general case}

To begin with, let us redefine the notion of (local) flux space, originally given in~\eqref{eq:d.flu.loc}. Let $\ell,k\in\N^\star$, and (a polytope) $K\in\KK$ be given. The (local) flux space $\dFlu(K)$ is henceforth defined as a polynomial subspace of
\[\vec{\Theta}^{\ell,k}(K)\defi\left\{\vec{\tau}\in\Hdiv[K]\mid\Div\vec{\tau}\in\Pol^{\ell-1}(K),\,\vec{\tau}_{\mid\pK}{\cdot}\vN_{\pK}\in\Pol^{k-1}(\FF_K)\right\}\]
satisfying that the map
\begin{equation} \label{eq:flux.map}
  \dFlu(K)\to\Gol^{\ell-2}(K)\times\Pol^{k-1}(\FF_K);\;\vec{\tau}\mapsto\big(\vprojGK(\vec{\tau}),\vec{\tau}_{\mid\pK}{\cdot}\vN_{\pK}\big)
\end{equation}
is surjective, with bounded right-inverse so that, for any couple $(\vec{\tau}_{\Gol,K},\tau_{\FF_K})\in\Gol^{\ell-2}(K)\times\Pol^{k-1}(\FF_K)$, there is $\vec{\tau}\in\dFlu(K)$ such that
\begin{equation} \label{eq:bnd.rinv}
  \|\vec{\tau}\|_K+h_K\|\Div\vec{\tau}\|_K\lesssim\|\vprojGK(\vec{\tau})\|_K+h_K^{\frac12}\|\vec{\tau}_{\mid\pK}{\cdot}\vN_{\pK}\|_{\pK}.  
\end{equation}
Note that, as opposed to definition~\eqref{eq:d.flu.loc}, elements in $\dFlu(K)$ are not anymore assumed to be curl-free.
In practice, for consistency reasons, it is also assumed that $\tb_K\Gol^{k-1}(K)\subseteq\dFlu(K)$. 
Let us denote by $\wprojFluK$ (resp.~$\wprojFluKK$) the $\tb_K^{-1}$-weighted $\vLL[K]$-orthogonal (resp.~$\tb^{-1}$-weighted $\vLL$-orthogonal) projector onto $\dFlu(K)$ (resp.~$\dFlu(\KK)$). 
We also let $\dPotb(K)$ be some (non-necessarily polynomial) finite-dimensional functional space such that the (DoF) map
\begin{equation} \label{eq:pot.map}
  \dPotb(K)\to\Pol^{\ell-1}(K)\times\Pol^{k-1}(\FF_K);\;v\mapsto\big(\projK(v),\projFFK(v_{\mid\pK})\big)
\end{equation}
is bijective (a valid choice is $\dPotb(K)=\dPot(K)$). In practice, again for consistency reasons, it is also assumed that $\Pol^k(K)\subseteq\dPotb(K)$. Importantly, $\tb_K\Grad\big(\dPotb(K)\big)$ is not assumed to be a subspace of the flux space $\dFlu(K)$.

Let us make a crucial observation: quite remarkably, the key formula~\eqref{eq:ibp} remains valid for all $v\in\dPotb(K)$ and $\vec{\tau}\in\dFlu(K)$ (with the new definition), yielding
\begin{equation} \label{eq:ibp.proj}
  \big(\tb_K^{-1}\wprojFluK(\tb_K\Grad v),\vec{\tau}\big)_K=-\big(\Div\vec{\tau},\projK(v)\big)_K+\big(\vec{\tau}_{\mid\pK}{\cdot}\vN_{\pK},\projFFK(v_{\mid\pK})\big)_{\pK}.
\end{equation}
In turn, the arguments from the proofs of Lemmas~\ref{le:d.ptom} and~\ref{le:d.mtop} trivially extend. One can thus infer the equivalence of the mixed method based on $\gFlu\defi\dFlu(\KK)\cap\Hdiv$ (with the new definition) and $\Pol^{\ell-1}(\KK)$ with the following primal problem: find $\til{x}_\KK\in\gPotb$ s.t.
\begin{equation} \label{eq:d.primal.proj}
  \big(\tb^{-1}\wprojFluKK\big(\tb\Grad_{\KK}\til{x}_{\KK}\big),\wprojFluKK\big(\tb\Grad_{\KK}\til{v}_{\KK}\big)\big)_{\Omega}=(f,\projKK(\til{v}_{\KK}))_{\Omega}\qquad\forall\til{v}_{\KK}\in\gPotb,
\end{equation}
where we have set $\gPotb\defi\dPotb(\KK)\cap\tilHzone$.
Phrased otherwise, $(\vec{\sigma},u)\in\gFlu\times\Pol^{\ell-1}(\KK)$ solves Problem~\ref{pb:d.mixed} if and only if
\[\vec{\sigma}=\wprojFluKK\big(\tb\Grad_{\KK}\til{x}_{\KK}\big)\qquad\text{and}\qquad u=\projKK(\til{x}_\KK),\]
where $\til{x}_\KK\in\gPotb$ is the unique solution to~\eqref{eq:d.primal.proj}.
Remark that, since $\dFlu(\KK)$ is assumed cell-wise polynomial, the projection $\wprojFluKK\big(\tb\Grad_{\KK}\til{x}_{\KK}\big)$ is always computable from~\eqref{eq:ibp.proj}. On the contrary, the function $\til{x}_\KK$ is only computable if $\dPotb(\KK)$ is cell-wise polynomial (or analytically known). In the opposite case, solely its DoFs (cf.~\eqref{eq:pot.map}) are accessible.
Primal methods of the form~\eqref{eq:d.primal.proj} are referred to as projection finite element methods in~\cite{ArbCh:95}. In the latter reference, it is implicitly assumed that the companion space $\dPotb(\KK)$ is computable. Here, we do not make such an assumption, i.e., $\dPotb(\KK)$ may remain purely virtual. Therefore, we will simply refer to these methods as projection methods (without reference to finite elements). Note that, even when the space $\dPotb(\KK)$ is virtual, it is possible to post-process, based on the DoFs, an optimally consistent (albeit fully discontinuous) potential reconstruction.
\begin{remark}[Well-posedness] \label{re:proj.wp}
  The (uniform) well-posedness of the mixed method based on $\gFlu$ (with the new definition) and $\Pol^{\ell-1}(\KK)$ can be established as in Remark~\ref{re:fdc.wp}, noticing that the surjectivity of the flux map~\eqref{eq:flux.map} is sufficient to prove the result, upon replacing the stability estimate~\eqref{eq:bnd.hdiv} by~\eqref{eq:bnd.rinv}.
  The (uniform) well-posedness of Problem~\eqref{eq:d.primal.proj}, in turn, is less obvious. As expected, it is, as well, a consequence of the surjectivity of the flux map~\eqref{eq:flux.map}. More precisely, combining the formula~\eqref{eq:div} with~\eqref{eq:ibp.proj}, for all $v\in\dPotb(K)$ and $\vec{\tau}\in\dFlu(K)$, there holds
  \begin{multline} \label{eq:ibp.proj.bis}
    \big(\tb_K^{-1}\wprojFluK(\tb_K\Grad v),\vec{\tau}\big)_K=\big(\vprojGK(\vec{\tau}),\Grad\projK(v)\big)_K\\+\big(\vec{\tau}_{\mid\pK}{\cdot}\vN_{\pK},\projFFK\big[\projFFK(v_{\mid\pK})-\projK(v)_{\mid\pK}\big]\big)_{\pK}.
  \end{multline}
  By surjectivity of the flux map~\eqref{eq:flux.map}, it is then possible to find $\vec{\tau}\in\dFlu(K)$ such that (i) $\vprojGK(\vec{\tau})=\Grad\projK(v)$ and (ii) $\vec{\tau}_{\mid\pK}{\cdot}\vN_{\pK}=h_K^{-1}\projFFK\big[\projFFK(v_{\mid\pK})-\projK(v)_{\mid\pK}\big]$, and satisfying the stability estimate~\eqref{eq:bnd.rinv}. We thus infer that, for all $K\in\KK$,
  \[\big\|\tb_K^{-\frac12}\wprojFluK(\tb_K\Grad v)\big\|_K\gtrsim\big\|\Grad\projK(v)\big\|_K+h_K^{-\frac12}\big\|\projFFK\big[\projFFK(v_{\mid\pK})-\projK(v)_{\mid\pK}\big]\big\|_{\pK},\]
  which implies that the map $\til{v}_{\KK}\mapsto\big\|\tb^{-\frac12}\wprojFluKK(\tb\Grad_{\KK}\til{v}_\KK)\big\|_{\Omega}$ is a norm on $\gPotb$, uniformly equivalent to $\|\tb^{\frac12}\Grad_{\KK}\til{v}_{\KK}\|_{\Omega}$.
\end{remark}
\noindent
The surjectivity of the flux map~\eqref{eq:flux.map} hence provides a simple criterion (to be checked in practice) for primal/mixed well-posedness and equivalence to hold true. This criterion is valid on general polytopal cells. It is essentially a reformulation (in simpler terms) of the criterion of~\cite[Thm.~4]{ArbCh:95} (note that the convexity assumption therein may be relaxed).

\begin{example}[Lowest-order Raviart--Thomas, revisited] \label{ex:ArBre}
  We have already seen above that, for any simplex $K\in\KK$, $\RT^1(K)=\tb_K\Grad(\dPot[1,1](K))$, where the space $\dPot[1,1](K)$ consists of the space $\Pol^1(K)$ enriched with a quadratic (weak) bubble. This construction leads to the equivalent primal formulation~\eqref{eq:d.primal:sim.oo}. In the seminal work~\cite{ArBre:85,Marin:85}, restricted to 2D, a different construction was actually advocated, leading to a space $\dPotb[1,1](K)$ with the same unisolvent set of degrees of freedom as $\dPot[1,1](K)$, but departing from $\dPot[1,1](K)$ (in particular, $\tb_K\Grad(\dPotb[1,1](K))\not\subset\RT^1(K)$). In those references, the space $\dPotb[1,1](K)$ is constructed by enriching $\Pol^1(K)$ with a cubic(/quartic in 3D) (strong) bubble.
  The resulting equivalent primal formulation~\eqref{eq:d.primal.proj} then writes: find $\til{x}_\KK\in\gPotb[1,1]$ such that
  \begin{equation} \label{eq:d.primal.proj:sim.oo}
    \big(\tb^{-1}\wprojFluxKK[1]{\RT}\big(\tb\Grad_{\KK}\til{x}_{\KK}\big),\wprojFluxKK[1]{\RT}\big(\tb\Grad_{\KK}\til{v}_{\KK}\big)\big)_{\Omega}=(f,\projKK[0](\til{v}_{\KK}))_{\Omega}\qquad\forall\til{v}_{\KK}\in\gPotb[1,1].
  \end{equation}
  We thus have that $(\vec{\sigma},u)\in\RT^1(\Omega)\times\Pol^0(\KK)$ solves Problem~\ref{pb:d.mixed} if and only if
  \[\vec{\sigma}=\wprojFluxKK[1]{\RT}\big(\tb\Grad_{\KK}\til{x}_{\KK}\big)=\tb\Grad_\KK\til{u}_\KK\qquad\text{and}\qquad u=\projKK[0](\til{x}_\KK)=\projKK[0](\til{u}_\KK),\]
  with $\til{x}_\KK\in\gPotb[1,1]$ and $\til{u}_\KK\in\gPot[1,1]$ unique respective solutions to~\eqref{eq:d.primal.proj:sim.oo} and~\eqref{eq:d.primal:sim.oo}.
\end{example}
\begin{example}[Arbitrary-order RT and BDM]
  Let a simplex $K\in\KK$ be given, and recall the definitions~\eqref{eq:RTN} of the local Raviart--Thomas and (first-kind) N\'ed\'elec spaces~\cite{RaTho:77b,Nedel:80}. Letting, for $n\in\N^\star$, $\BDM^n(K)\defi\vPol^n(K)$ denote the local Brezzi--Douglas--Marini space~\cite{BrDoM:85,Nedel:86}, it is known that the two (DoF) maps
  \begin{itemize}
    \item[$\bullet$] $\BDM^n(K)\to\Ne^{n-1}(K)\times\Pol^n(\FF_K);\;\vec{\tau}\mapsto\big(\projFluxK[n-1]{\Ne}(\vec{\tau}),\vec{\tau}_{\mid\pK}{\cdot}\vN_{\pK}\big)$,
    \item[$\bullet$] $\RT^n(K)\to\vPol^{n-2}(K)\times\Pol^{n-1}(\FF_K);\;\vec{\tau}\mapsto\big(\vprojK[n-2](\vec{\tau}),\vec{\tau}_{\mid\pK}{\cdot}\vN_{\pK}\big)$,
  \end{itemize}
  are bijective (note that $\Ne^0(K)=\{\vec{0}\}$). Fix now an integer $k\geq 2$, and set $\dFlu[k-1,k](K)\defi\BDM^{k-1}(K)$ and $\dFlu[k,k](K)\defi\RT^k(K)$. Then, one can easily verify that the space $\dFlu(K)$ for $\ell\in\{k{-}1,k\}$ is a (polynomial) subspace of $\vec{\Theta}^{\ell,k}(K)$ satisfying that the flux map~\eqref{eq:flux.map} is surjective (with bounded right-inverse so that~\eqref{eq:bnd.rinv} holds).
  Remark, indeed, that $\Gol^{k-3}(K)\subseteq\Ne^{k-2}(K)$ and $\Gol^{k-2}(K)\subseteq\vPol^{k-2}(K)$.
  One can thus prove equivalence, for $\ell\in\{k{-}1,k\}$, between the mixed (BDM or RT) method based on $\gFlu\defi\dFlu(\KK)\cap\Hdiv$ and $\Pol^{\ell-1}(\KK)$ and the primal (projection) method~\eqref{eq:d.primal.proj} based on $\gPotb$ (for any choice of $\dPotb(K)$ such that the potential map~\eqref{eq:pot.map} is bijective). In~\cite[Sec.~4]{ArbCh:95}, building on ideas from~\cite{ArBre:85}, polynomial companion spaces $\dPotb(K)$ are constructed in 2D for both BDM and RT (note, however, that for $k$ even, the latter spaces do not contain $\Pol^k(K)$). An alternative path is to consider a non-computable companion space, for instance $\dPot(K)$. In such a scenario, the companion space is purely virtual. This is the approach pursued (without even introducing the notion of companion space) by self-stabilized (also termed unstabilized or stabilization-free) hybrid high-order (or weak Galerkin) methods (still on simplices). We refer the reader to~\cite{AbErP:18} (cf.~also~\cite[Sec.~3.1]{CiEPi:21}) for RT and to~\cite{CaTra:21} (cf.~also~\cite{CDPLe:22}) for (RT and) BDM. We make completely clear in this work that such approximations are actually equivalent to the RT and BDM (mixed) finite element schemes. 
\end{example}
\begin{remark}[Polytopal cells]
  First remark that, for our theory above to be valid (in particular, primal/mixed equivalence), one does not need to assume that the local flux space, subset of $\vec{\Theta}^{\ell,k}(K)$, be polynomial, as long as it remains computable (for instance, piecewise polynomial is fine). 
  Self-stabilized hybrid high-order (or weak Galerkin) methods have also been studied on general polytopal cells. A common approach is to consider a partition $\mathcal{S}_K$ of the cell $K$ into sub-simplices. In~\cite{DPDMa:18}, for $\ell\in\{k{-}1,k,k{+}1\}$, a stable gradient reconstruction is obtained in each cell $K$ as the sum of a consistent contribution in $\vPol^{k-1}(K)$ (or $\Gol^{k-1}(K)$) and of an orthogonal stabilization in the broken space $\RT^{k+1}(\mathcal{S}_K)$ (or $\RT^{\max(\ell,k)}(\mathcal{S}_K)$). However, such a reconstruction does not belong to $\Hdiv[K]$ (and thus to $\vec{\Theta}^{\ell,k}(K)$) in general, hence it does not enter the present framework. In contrast, in~\cite{YeZha:21a,YeZha:21b}, a gradient reconstruction in $\RT^k(\mathcal{S}_K)\cap\vec{\Theta}^{\ell,k}(K)$ (resp.~$\vPol^{k-1}(\mathcal{S}_K)\cap\vec{\Theta}^{\ell,k}(K)$) is proposed for $\ell=k$ (resp.~$\ell=k{-}1$). If the surjectivity of the flux map~\eqref{eq:flux.map} is established therein, yielding the unique solvability of the corresponding projection methods, no proof of~\eqref{eq:bnd.rinv} is however provided. It is thus not clear whether the proposed constructions yield uniformly stable methods. A similar reconstruction strategy has also been studied in~\cite{CaQDP:24}, albeit in the Navier--Stokes context (the vectorial variable to be reconstructed thus belongs to $\vHone$). Subtessellation-avoiding approaches are also available. However, none of them enters the present framework. In~\cite{YeZha:20}, a gradient reconstruction in $\vPol^{\car(\FF_K)+k-2}(K)(\not\subset\vec{\Theta}^{\ell,k}(K))$ is studied for $\ell=k$, which is however suboptimally consistent. In~\cite{BoCCM:25}, in turn, an alternative approach is proposed for $\ell\in\{k{-}1,k,k{+}1\}$, with gradient reconstruction in $\Grad\big(\mathcal{Q}^n(K)\big)(\not\subset\vec{\Theta}^{\ell,k}(K))$, where $\mathcal{Q}^n(K)\defi\{v\in\Pol^n(K)\mid\triangle p\in\Pol^{\ell-1}(K)\}$ for $n>\ell+1$. The reconstructed gradient is optimally consistent, and stability is observed numerically (but not proven) for $n$ such that $\dim\big(\mathcal{Q}^n(K)\big)\geq\dim\big(\Pol^{\ell-1}(K)\big)+\dim\big(\Pol^{k-1}(\FF_K)\big)$. Primal/mixed equivalence for the latter approaches (as well as for~\cite{DPDMa:18}) will be investigated in Part II of this work. In that case, some kind of equivalence shall be recovered, but at the price of $\Hdiv$-conformity.
\end{remark}

\section{Conclusion}

In this article, we have studied primal/mixed equivalence on general polytopal partitions, building additional bridges between (oversampling-free) multiscale methods, and shedding a new light on projection methods, with implications on self-stabilized hybrid methods. Our study was performed under the fundamental assumption that the global flux space is conforming in $\Hdiv$ which, provided the latter space is computable, guarantees that the corresponding mixed formulation is a (mixed) finite element method.
This crucial assumption will be relaxed in Part II of this work, dedicated to (general) polytopal element methods.

\section*{Acknowledgments}

The work of the author is supported by the Agence Nationale de la Recherche (ANR) under the PRCE grant HIPOTHEC (ANR-23-CE46-0013).
The author also acknowledges the support of the CDP C2EMPI, together with the French State under the France-2030 programme, the University of Lille, the Initiative of Excellence of the University of Lille, the European Metropolis of Lille for their funding and support of the R-CDP-24-004-C2EMPI project.

\ifSINUM

\bibliographystyle{siamplain}
\bibliography{equiv}

\else

\bibliographystyle{plain}
{\small
\bibliography{equiv}
}

\fi

\end{document}